%% file: current_version__4_.tex
\documentclass[reqno]{amsart}
\usepackage[T1]{fontenc}  
\usepackage[utf8]{inputenc} 
\input{packages}

\input{commands}

\input{theoremstyle}

\begin{document}

\title{The logarithmic leaf complex and foliated d-semistability}

\author[M. Corr\^ea]{ Maur\'icio Corr\^ea }
\address[M. Corr\^ea]{Dipartimento di Matematica, Universit\`a degli Studi di Bari,  Via E. Orabona 4, I-70125, Bari, Italy
}
\email{mauricio.correa.mat@gmail.com, mauricio.barros@uniba.it}

\author[ P. Perrella]{  Pablo Perrella }
\address[P. Perrella]{Departamento de Matem\'atica, Facultad de Ciencias Exactas y Naturales, Universidad de Buenos Aires and IMAS (UBA-CONICET), Buenos Aires, Argentina.
}

\email{pperrella@dm.uba.ar }

\author[ S. Velazquez]{ Sebasti\'an Velazquez }
\address[S. Velazquez]{Department of Mathematics, King’s College London, Strand, London WC2R 2LS, United Kingdom}

\email{sebastian.velazquez@kcl.ac.uk}



\begin{abstract} 
We study holomorphic foliations on normal crossings varieties arising as semistable degenerations. We do so by we exploring the notion of foliated $d$-semistability using the language of  logarithmic structures in the sense of Fontaine–Illusie. First, we  identify both local and global obstructions to $d$-semistability. In order to analyze the existence of smoothings, we develop a logarithmic deformation theory of foliations, and show that the corresponding moduli functor admits a versal hull.
\end{abstract}

\maketitle

\section{Introduction}

The main goal of this paper is to introduce the notions of $d$-semistability and log smoothness in the context of foliated pairs and their moduli.

In light of recent advances in the moduli theory of foliations \cite{spicer2025moduli}, it has become increasingly clear that a satisfactory understanding of foliations on  varieties with normal crossings and their deformation theory is essential, even if one was primarily interested in foliations on smooth varieties, since these objects arise naturally at the boundary of moduli problems. Moreover, given a family of foliated pairs $(\X,\F)\to \Delta$ whose general fiber is smooth, the Semistable Reduction Theorem produces, after a finite base  $\Delta'\to \Delta$, a birational model
\begin{center}
    \begin{tikzcd}
        \widetilde{\X}\arrow[r]\arrow[d] &\arrow[d] \X\times_\Delta \Delta' \\
        \Delta' \arrow[r]& \Delta'
    \end{tikzcd}
\end{center}
with smooth total space, reduced normal crossings central fiber, and agreeing with the pulled-back family away from the central fiber. Endowing this new model with the pullback foliation, one is naturally led to study the induced foliated structure on the central fiber and its \emph{smoothing} \smash{$(\widetilde{X},\widetilde\F)\to \Delta$}. In this way, the problem of understanding families whose generic fiber is a smooth variety naturally leads to the study of \say{smoothings} of foliations on normal crossing varieties.
This motivates the following: 

\textbf{Question:}\, \,which foliations on varieties with normal crossings admit smoothings?

In the non-foliated setting, this question was first addressed in \cite{friedman1983global}, where the notion of $d$-semistability was introduced. It is not difficult to see that a variety admitting a smoothing as above is necessarily $d$-semistable; however, determining when the converse holds has led to a substantial body of interesting mathematics. These ideas were later further developed through the lens of logarithmic geometry (see \cite{kawamata1994logarithmic,kato1988logarithmic_str,kato1996dlog_def}), which has proven to be particularly well-suited to this problem.

Loosely speaking, logarithmic geometric 
can be viewed as an enriched version of algebraic geometry in which one equips a ringed space $\underline{X}$ with a sheaf of monoids $\M_X$ and a morphism
$$ \alpha:\M_X\to \OO_X,$$
where the latter is regarded as a sheaf of monoids under multiplication. When no confusion arises, we will use the notation $\underline{X}$ for the underlying space and reserve the symbol $X$ for the logarithmic space $(\underline{X},\M_X)$. This notion was first introduced in \cite{kato1988logarithmic_str}, and gives rise to a geometric category closely parallel to that of classical algebraic geometry.
In this setting, there is a notion of log smoothness which behaves, in many respects, analogously to the usual notion of smoothness, allowing one to extend a wide range of results and geometric intuition to more singular spaces.
In the case of a variety with normal crossings, the existence of certain log smooth morphisms $X\to (\ast, \N)$ to the standard log point (see Example \ref{ex: log point}) is equivalent to $\underline{X}$ being $d$-semistable, thereby illustrating the convenience of this approach. 

To a morphism of logarithmic spaces $X\to S$ one can associate a relative tangent sheaf $T_{X/S}$, endowed with a Lie bracket and a forgetful morphism of sheaves of Lie algebras $T_{X/S}\to T_{\underline{X}/\underline{S}}$.
In the case of a semistable morphism $X\to (\ast,\N)$, the forgetful map is an injection and hence it makes sense to say that a foliation $\F$ on $\underline{X}$ is $d$-semistable if $T_{\F}\subseteq T_{X/(\ast,\N)}\subseteq T_{\underline{X}}$.
It is important to point out that, while varieties with normal crossings are automatically \say{locally semistable} \cite[Prop. 11.3]{kato1988logarithmic_str}, this is no longer the case once a foliation is added. Indeed, we show that local semistability imposes strong compatibility conditions on the foliations defined on the irreducible components.

The paper is organized as follows. We begin by recalling the notions from logarithmic geometry that will be needed in what follows.

Section 3 is devoted to the analysis of $d$-semistability. On the one hand, we exhibit different instances of local $d$-semistability and establish a compatibility relation for the Camacho–Sad indices of the induced foliations on the irreducible components.
With this in place, we turn to the study of global $d$-semistability. We show that the obstruction to global $d$-semistability lies in the cohomology group
$$H^1\big(\underline{X},\OO^\ast_{\underline{X}/\F}/G_\F\big),$$
where \smash{$\OO^\ast_{\underline{X}/\F}$} is the sheaf of invertible first integrals of $\F$ and $G_\F$ the kernel of the morphism $\OO^\ast_{\underline{X}/\F}\to \OO_D^\ast$. Among other consequences, this shows that if the double locus is a leaf of $\F$, then the obstruction to $d$-semistability is of topological nature.

In Section 4 we develop  a log-smooth deformation theory for singular log foliations. Inspired by the works of \cite{kato1996dlog_def,gomez1988transverse}, we show that the deformation theory of foliations on log smooth varieties is governed by a logarithmic analogue of the leaf complex and prove that versal hulls exist, provided that $\underline{X}$ is proper.

Combining the results of Sections 3 and 4, we construct families of foliations that admit smoothings and lay the groundwork for a systematic study of this problem.
\medskip

\textbf{Acknowledgements} The first author was supported by CONICET, Argentina. 
MC was supported by the PRIN 2022MWPMAB ``Interactions between Geometric Structures and Function Theories'', and by INdAM–GNSAG. The third author was supported by EPSRC, United Kingdom. Part of this work was carried out during a visit of the first author to King's College London, funded by the London Mathematical Society, to whom we are very grateful.

\section{Preliminaries}
\subsection*{Conventions}

Along the way all monoids are assumed to be commutative and will have a neutral element. We will denote by $P^\gp$ the Grothendieck group of $P$. We write $\N$ for additive monoid of non-negative integer. 

\begin{definition}
A monoid $P$ is said to be
\begin{enumerate}
    \item \textit{finitely generated} if there is a surjective homomorphism $\N^k\rightarrow P$ for some $k\geqslant 1$,
    \item \textit{integral} if the natural homomorphism $P\rightarrow P^\gp$ is injective,
    \item \textit{saturated} if $P$ is integral and for each $x\in P^\gp$ such that $x^k\in P$ for some integer $k\geqslant 1$, then $x\in P$.
\end{enumerate}
\end{definition}

In order to simplify verbose statements, from now on all monoids will be finitely generated, integral and saturated. 

On the other hand, all our spaces will be complex analytic and the sheaves on them will be considered with respect to their analytic topology.

\subsection{Logarithmic structures}
We will begin by introducing the general notions on logarithmic geometry necessary for our purposes. For more complete expositions, the reader is referred to \cite{kato1988logarithmic_str, kato1996dlog_def, abramovich2010logarithmic}.

\begin{definition} \label{def:logscheme}
A \textit{pre-log structure} on a complex analytic space $\underline{X}$ is a sheaf of monoids $\M_X$ on $\underline{X}$ and a multiplicative morphism of sheaves $\alpha: \M_X\rightarrow \OO_{\underline{X}}$. If the restriction $\alpha^{-1}(\OO^\ast_{\underline{X}}) \rightarrow\OO^\ast_{\underline{X}}$ is an isomorphism, we will say that $\M_X$ is a \textit{log structure}. The pair $(\underline{X},\M_X)$ will be called a \textit{log space}, and will be denoted by $X$ to distinguish it from its underlying space $\underline{X}$. We shall use the notation $\OO_X$ instead of $\OO_{\underline{X}}$ for simplicity.
\end{definition}

\begin{definition}\label{def:morphism}
    A \emph{morphism of log spaces} $f:X\rightarrow Y$ consists of a morphism between the underlying complex analytic spaces $f: \underline{X}\rightarrow \underline{Y}$ together with a morphism of sheaves $f^\flat: f^{-1} \M_Y\rightarrow \M_X$ making the diagram
\begin{center}
    \begin{tikzcd}
        f^{-1}\M_Y\arrow[r]\arrow[d] &\arrow[d] \M_X \\
        f^{-1}\OO_Y \arrow[r]& \OO_X
    \end{tikzcd}
\end{center}
commute. Two log structures $\M_X$ and $\M'_X$ on $\underline{X}$ are \textit{equivalent} if there exists an isomorphism of log spaces $(\underline{X},\M_X)\to (\underline{X},\M'_X)$, i.e. an isomorphism of log spaces witch is the identity on $\underline{X}$.
\end{definition}

\begin{example}
Complex analytic spaces $\underline{X}$ can be naturally be viewed a log spaces. In fact, the inclusion $\M_{X} :=\OO_{\underline{X}}^\ast\subseteq \OO_{X}$ is a log structure.
\end{example}

\begin{example}\label{ex:log_str_div}
Given a hypersurface $D$ on a complex manifold $\underline{X}$, the subsheaf of monoids $\M_{(X,D)}\subseteq\OO_X$ given by those local functions $f$ that do not vanish at any point of $\underline{X}\setminus D$ is a log structure. We will denote the corresponding log space by $(X,D)$.
\end{example}

Given a pre-log structure $\alpha:\M_X\rightarrow \OO_X$ we can construct its \textit{associated log structure} $\M_X^a\rightarrow \OO_X$ as the log structure obtained as the push-out diagram

\begin{center}
\begin{tikzcd}
\alpha^{-1}(\OO_X^\ast) \arrow[r] \arrow[d]          & \M_X \arrow[d] \arrow[rdd, bend left] &       \\
\OO_X^\ast \arrow[r] \arrow[rrd, bend right] & \M^a_X \arrow[rd]                     &       \\
                                                     &                                             & \OO_X.
\end{tikzcd}
\end{center}

This procedure allow us to construct new  logarithmic structures in the following way. Given a monoid $P$, a homomorphism $P\rightarrow \Gamma(\underline{X}, \OO_{\underline{X}})$ induces a pre-log structure $P_{X}\rightarrow \OO_{\underline{X}}$ where $P_X$ is a constant sheaf. Therefore we obtain an associated log structure $(P_X)^a\rightarrow \OO_X$ as before.

\begin{example}\label{ex: log point}
Given a monoid $Q$, the homomorphism $\alpha: Q\rightarrow \C$ defined by $\alpha(x) = 0$ for each $1\neq x\in Q$, and $\alpha(1) = 1$ produces a logarithmic structure on the point $\ast = \Spec \C$ that will be denoted by $(\ast, Q)$.  The \textit{standard log point} is the log space $(\ast,\N)$. As a matter of fact, any log structure on $\ast$ is equivalent to one of this form \cite{kato1988logarithmic_str}.
\end{example}

\begin{example}\label{ex: ss_log_model}
This will be the leading log structure of this work. The homomorphism $\N^r\rightarrow \C\{x_1,\dots,x_n\}/(x_1\cdots x_r)$ defined on the canonical basis of $\N^r$ by $e_i \mapsto \overline{x}_i/u_i$ for some units $u_i$ defines a log space $X$ whose underlying space is the simple normal crossing germ $\underline{X} = \Spec \C\{x_1,\dots,x_n\}/(x_1\cdots x_r)$.
\end{example}

\begin{example}\label{ex: monoid_log_str}
Given a commutative analytic ring $R$ and a monoid $P$, there is a canonical homomorphism $P\rightarrow R[P]$ to the monoid algebra $R[P]$ that produces a log structure on $\underline{X}=\Spec R[P]$. 
\end{example}

\begin{example}\label{ex: toric_log_str}
Every toric variety $\underline{X}_\Sigma$ associated to a fan $\Sigma \subseteq N \otimes_{\mathbb{Z}} \mathbb{R}$ carries a natural logarithmic structure $\mathcal{M}_{X_\Sigma}$. Indeed, this variety admits an open cover by affine toric varieties of the form $U_\sigma = \Spec \mathbb{C}[M \cap \sigma^\vee]$, where $\sigma \in \Sigma$ and $M = \Hom(N,\mathbb{Z})$ is the dual lattice. The logarithmic structures on each $U_\sigma$, induced by the monoid algebras as in the previous example, glue to define $\mathcal{M}_{X_\Sigma}$.
\end{example}

\begin{definition}
A \textit{chart} of a log space $X$ is a morphism $P\rightarrow \Gamma(X,\M_X)$ with $P$ a monoid such that the associated map $(P_X)^a\rightarrow \M_X$ of log structures is an isomorphism. If $X$ admits local charts at around each point we will say that $X$ is a \textit{fine log space}.
\end{definition}

All the examples of log spaces presented so far are fine log spaces. Consider for instance the log space $(X,D)$ of induced by a divisor as in Example \ref{ex:log_str_div}. Locally on $\underline{X}$, the hypersurface $D$ can be described as the zero locus of some prime elements $f_1,\dots, f_k\in \Gamma(U,\OO_{X})$. The morphism $\N^k \rightarrow \Gamma(U,\OO_X)$ induced by those functions produce the desired chart on $\M_{(X,D)}|_U$.

\begin{definition}
Let $f: \underline{X}\rightarrow \underline{Y}$ be a morphism of complex analytic spaces. Given a log structure $\M_Y$ on $\underline{Y}$, we the define the \textit{pullback} $f^\ast\M_Y$ as the log structure on $\underline{X}$ associated to the pre-log structure $f^{-1}\M_Y\rightarrow f^{-1}\OO_Y\rightarrow \OO_X$. A morphism $f: X \rightarrow Y$ of log spaces is \textit{strict} if $f^{\flat}: f^\ast\M_Y\rightarrow\M_X$ is an isomorphism.
\end{definition}

\subsection{Semistable logarithmic structures} In his seminal paper \cite{friedman1983global}, Friedman studied what properties inherit a semistable degeneration of smooth varieties. He recognized that the special fiber should satisfy certain topological and analytical constraints. Hence he introduced the notion of \textit{d-semistablity} and proved that d-semistable K3 surfaces are semistable degenerations. With a deformation theoretical approach, Kawamata and Namikawa extended this result to d-semistable Calabi-Yau varieties \cite{kawamata1994logarithmic}. They recognized that $d$-semistable varieties are, in terms of logarithmic geometry, those having a special type of logarithmic structure (see Theorem \ref{thm: semistability_T1}
 below).

\begin{definition}\label{def: semistable str}
Let $\underline{X}$ be a variety with normal crossings. A logarithmic structure $\M_X$ on $\underline{X}$ is of \textit{semistable type} if there exists a morphism of log spaces $f:(\underline{X},\M_X)\rightarrow (\ast,\N)$ such that locally on $X$ we can find charts $\N^r\rightarrow \Gamma(U,\M_X)$ making the diagram
\begin{center}
\begin{tikzcd}
\N \arrow[r] \arrow[d] & {\Gamma(\ast,\M_{(\ast,\N)})} \arrow[d, "f^\flat"] \\
\N^r \arrow[r]         & {\Gamma(U,\M_X)}                              
\end{tikzcd}
\end{center}
\end{definition}
commute, with the left vertical arrow being the diagonal map. Such morphism $f$ will be called a \textit{semistable map}. We will say that $\underline{X}$ is \textit{d-semistable} if it admits a logarithmic structure of semistable type.

From the definition itself, determining whether a given normal crossings variety admits a semistable logarithmic structure appears to be a subtle problem. Before moving on, we present a practical criterion to construct such structures in terms of the module \smash{$T^1_{\underline{X}} = \Ext^1(\Omega^1_{\underline{X}},\OO_{\underline{X}})$}. 

Let $\underline{X}=\Spec\C\{x_1,\dots,x_n\}/(x_1\dots x_r)$ be a normal 
crossing germ with singular locus $D$. Denoting $X_i=V(x_i)\subseteq (\C^n,0)$, we have that  the sheaves $I_{X_i}\vert_D$ are locally free on $D$ and generated by $\overline{x}_i$. Hence the sheaf 
$I_{X_1}\vert_D\otimes\dots\otimes I_{X_r}\vert_D$
is also locally free of rank one and generated by $\overline{x}_1\otimes \dots\otimes\overline{x}_r$. By a straightforward  computation \cite[Prop. 2.3]{friedman1983global}, there exists a canonical isomorphism $$T^1_{\underline{X}}\big\vert_D\simeq\big(I_{X_1}\vert_D\otimes\dots\otimes I_{X_r}\vert_D\big)^\vee.$$

\begin{theorem}[{{\cite[Prop. 1.1]{kawamata1994logarithmic}}}]{\label{thm: semistability_T1}} 
A normal crossing variety $\underline{X}$ with singular locus $D$ is d-semistable if and only if $T_{\underline{X}}^1\big\vert_D \simeq \OO_D$.    
\end{theorem}

We briefly recall the proof, as we will need it in the next section. Assume first that $\underline{X}$ admits a semistable log structure $\M_X$ and take a semistable map $f:X\to (\ast, \N)$. Let  $\N^r\to \Gamma(U,\M_X)$ be a local chart as in Definition \ref{def: semistable str}, and denote $x_i=\alpha(e_i)$. We claim that the local sections $\overline{x}_1\otimes \cdots \otimes \overline{x}_r$ constructed on each $U$ glue to a nowhere-vanishing global section of $T^1_{\underline{X}}$. Indeed, a different choice of chart $\N^r\to \M_X$ will induce some element  $\overline{x}_1'\otimes \cdots \otimes \overline{x}_r'$ such that $x_i=u_i x_i'$ and $\prod_{i=1}^r u_i=1$ (by the good definition of the map $f^b$ above), showing that the constructed sections of $T^1_{\underline{X}}$ coincide on the overlaps. This proves the first implication. Reading the argument backwards we see that an isomorphism \smash{$(T^1_X\big\vert_D)^\vee\simeq \OO_D$} induces a logarithmic structure of semistable type.

\begin{example}[\cite{friedman1983global}]
Consider the \say{tetrahedron} $\underline{X}\subseteq \P^3$ obtained as the union of 4 planes $X_i\subseteq \P^3$ in general position. This variety is no d-semistable because $T_{\underline{X}}^1 = I_{X_1}|_D\otimes I_{X_2}|_D\otimes I_{X_3}|_D\otimes I_{X_4}|_D \simeq \OO_D(4)$. However, the blow up of $\underline{X}$ along 4 general points at each of the 6 irreducible components of $D$ does admits such structure.
\end{example}

\subsection{Log derivations and differential forms}
\begin{definition} \label{def:logderivation}
Let $f:X\to S$ be a morphism of  log spaces. The sheaf $T_{X/S}$ of \textit{relative log derivations} is the sheaf of germs of pairs $\theta = (v,\delta)$ where $v: \OO_{X}\rightarrow \OO_{X}$ is a $f^{-1}\OO_S$-linear derivation and $\delta:\M_X\to \OO_X$ is such that
\begin{enumerate}
    \item $\delta(ab)=\delta(a)+\delta(b)$,
    \item $\alpha(a)\delta(a)= v(\alpha(a))$, 
    \item $\delta\circ f^{\flat}=0$.
\end{enumerate}
\end{definition}
In the logarithmic setting, the sheaf $T_{X/S}$ in endowed with the relative Lie bracket that takes two log derivations $\theta_1 = (v_1,\delta_1)$ and $\theta_2 = (v_2,\delta_2)$ and gives
    $$ \big[\theta_1,\theta_2\big]:=\big([v_1,v_2], v_1\delta_2-v_2\delta_1\big),$$
where $[v_1,v_2]$ is the relative Lie bracket of vector fields on $\underline{X}/\underline{S}$  \cite[Prop. 1.1.2]{ogus2018lectures}.

\begin{definition}
The sheaf of log differentials relative to a morphism $f:X\rightarrow S$ is the quotient 
$$\Omega^1_{X/S}=\big(\Omega^1_{\underline{X}/\underline{S}} \oplus (\OO_X\otimes_\Z \M_X^{gp})\big)/K, $$
where $K$ is the $\OO_X$-module generated by the expressions $(d\alpha(a),0)-(0,\alpha(a)\otimes a)$ for $a\in \M_X$ and $(0,1\otimes \varphi(b))$ for $b\in f^{-1}\mathcal{M}_S$.    
\end{definition}

If the base $S$ is the trivial log point then we will remove part of the subscript out of $T_{X/S}$ and $\Omega^1_{X/S}$, and write instead $T_X$ and $\Omega^1_X$ respectively since those are the usual tangent and cotangent sheaves of a space.

\begin{example}
Consider the log space $(X,D)$ associated to a normal crossing divisor as in Example \ref{ex:log_str_div}. The sheaf of logarithmic differentials of $(X,D)$ is $\Omega^1_{(X,D)} = \Omega^1_{\underline{X}}(\log D)$, famously introduced by Deligne \cite{deligne1970equations}. Similarly, $T_{(X,D)} = T_{\underline{X}}(-\log D)$ is the sheaf of derivations $v$ such that $v(I_{D})\subseteq I_D$.
\end{example}

\begin{example} \label{ex: tangente_toric}
Suppose $X_\Sigma$ is the log space associated to a $\Q$-factorial toric variety without torus factors. Then we have isomorphisms
\[
\Omega^1_{X_\Sigma}\simeq \OO_{X_\Sigma}\otimes N \quad\text{and}\quad T_{X_\Sigma}\simeq \OO_{X_\Sigma}\otimes M,
\]
where $M$ and $N$ are the character and 1-parameter subgroup lattices of $X_\Sigma$ respectively.
\end{example}

\begin{example}\label{ex:tangent_underline}
Let $\underline{X} = \Spec \C\{x_1,\dots,x_n\}/(x_1\cdots x_r)$ be endowed with log structure $\M_X$ defined by $\N^r\to \OO_X$ sending $e_i\mapsto \overline{x}_i/u_i$ as in Example \ref{ex: ss_log_model}. If we write $\hat{x}_i = x_1\cdots \hat{x}_i\cdots x_n$, then the usual (non-logarithmic) module of Kähler differentials is
\[
\Omega^1_{\underline{X}} = \OO_{\underline{X}} dx_1\oplus\cdots \oplus\OO_{\underline{X}} dx_n / (\widehat{x}_1dx_1+\cdots+\widehat{x}_rdx_r)\\
\] and therefore the tangent module is generated by the following derivations
\begin{align*}
T_{\underline{X}} = \big\langle x_1\partial_{x_1},\dots,x_r\partial_{x_r},\partial_{x_{r+1}},\dots, \partial_{x_n}\big\rangle\subseteq \OO_{\underline{X}} \partial_{x_1}\oplus\cdots \oplus\OO_{\underline{X}} \partial_{x_n}.\\
\end{align*}
On the other hand, we can also compute differentials and derivations relative to the morphism $\underline{X}\rightarrow (\ast,\N)$ associated to the diagonal map $\N\rightarrow \N^r$. We denote by $u = u_1\cdots u_r$ the product of the former units. Therefore
\begin{align*}
\Omega^1_{X/(\ast,\N)} & =\bigoplus_{i=1}^r\OO_{X} \frac{dx_i}{x_i}\oplus\bigoplus_{i=r+1}^n\OO_{X} dx_{i}\Big/\Big (\frac{dx_1}{x_1}+\cdots+\frac{dx_r}{x_r}-\frac{du}{u}\Big),\\
T_{X/(\ast,\N)} & = \bigg\{v=\sum^r_{i=1} b_i x_i\partial_{x_i} + \sum^n_{i=r+1} a_i\partial_{x_i}  \in T_{\underline{X}}: v(u) - \bigg(\sum^r_{i=1} b_i\bigg)u=0\bigg\}.
\end{align*}
Remarkably, the logarithmic (co)tangent sheaves computed so far are locally free, even for log spaces whose underlying spaces are quite singular.  This is an incarnation of smoothness in the logarithmic context. In fact, this one of the main reasons to work with log spaces: they provide a broader class of smooth objects.
\end{example}
\begin{definition}
A morphism of log spaces $T\rightarrow T'$ is an \textit{infinitesimal thickening} if it is strict and the underlying morphism $\underline{T}\rightarrow \underline{T}'$ is a closed inmersion with nilpotent sheaf of ideals $I_{T/T'}$. The \textit{order} of this thickening is the minium $k\geqslant 1$ such that $I_{T/T'}^k = 0$.
\end{definition}

\begin{definition}
A morphism $f: X\rightarrow Y$ between fine log spaces is \textit{log smooth} if the underlying map $f:\underline{X}\rightarrow \underline{Y}$ is locally of finite presentation and it satisfies the following \textit{infinitesimal lifting property}: given a diagram of solid arrows
\begin{center}
\begin{tikzcd}
T \arrow[d, hook] \arrow[r]   & X \arrow[d,"f"] \\
T' \arrow[r] \arrow[ru, dotted] & Y          
\end{tikzcd}
\end{center}
where the left vertical arrow is an infinitesimal thickening, locally there exist a dotted arrow making the previous triangles commute.
\end{definition}

\begin{remark}
    A fine saturated log scheme, smooth over the trivial log point corresponds to a smooth variety $X$ equipped with the logarithmic structure of a normal crossings divisor $D$. Over the non-trivial log point, on the other hand, a semistable map $X\to (\ast,\N)$ is log smooth (see \cite{kato1988logarithmic_str,kato1996dlog_def}).
\end{remark}

\section{Log foliations and $d$-semistability}

\begin{definition}
Let $X\rightarrow (\ast,Q)$ be a log smooth morphism from a fine log space $X$ into the log point associated to some monoid $Q$. A \textit{log (singular) foliation} $\F$ of codimension $q$ on $X/(\ast,Q)$ consists of a short exact sequence of coherent sheaves
\[
0\rightarrow T_{\F}\rightarrow T_{X/(\ast,Q)}\rightarrow N_\F\rightarrow 0
\]
such that $N_\F$ is torsion-less of rank $q$ and $[T_\F,T_\F]\subseteq T_\F$. We will call $T_{\F}$ and $N_{\F}$ respectively the \textit{tangent} and \textit{normal sheaf} of the foliation $\F$. 
\end{definition}

Recall that a coherent sheaf $N$ is \textit{torsion-less} if the natural map $N\rightarrow N^{\vee\vee}$ is a monomorphism. This is a replacement of the usual torsion free hypothesis in the singular setting. This allow us to equivalently represent a log foliation by an exact sequence of coherent sheaves
\[
0\rightarrow I_\F\rightarrow \Omega^1_{X/(\ast,Q)}\rightarrow \Omega^1_\F\rightarrow 0
\]
where $\Omega^1_\F$ is torsion-less and, on a neighborhood of all general points of $\underline{X}$, there is a local basis $\omega_1,\dots, \omega_q\in I_\F$ such that $d\omega_i \wedge \omega_1\wedge\cdots \wedge \omega_q$ for all $1\leq i\leq q$ \cite[Lemma 4.2]{Quallbrunn_2015}.

\begin{example}
Let $\underline{X}$ be a smooth manifold endowed with the log structure $\M_{(X,D)}$ associated with a simple normal crossing divisor $D = D_1\cup\cdots\cup D_r$.  Let us take global sections $f_i\in\Gamma(X,\OO_X(D))$ and non-zero complex numbers $\lambda_i$ such that $\sum_{i=1}^r \lambda_i [D_i] = 0$ in $H^2(\underline{X},\C)$. As in \cite{calvo1994irreducible} we can construct an integrable twisted 1-form $$\omega = f_1\cdots f_r \sum_{i=1}^r \lambda_i \frac{df_i}{f_i}$$
defining a foliation on $\underline{X}$. On the other hand, $\omega/(f_1\cdots f_r)$ is a global section of $\Omega^1_X = \Omega^1_{\underline{X}}(\log D)$ and hence defines log foliation $\F$ on the log space $X$. Since this section is nowhere-vanishing, the foliation $\F$ is smooth in the logarithmic sense. 
\end{example}

\begin{example}
Consider a toric variety $\underline{X}_\Sigma$ obtained as the compactification of a torus $T$ with respect to a fan $\Sigma$, and let $N$ be the lattice of 1-parameter subgroups of $T$. A foliation $\F$ on $\underline{X}_\Sigma$ is \textit{toric} if for every $t\in T$ we have $t^\ast T_\F\simeq T_\F$ under the identification $t^\ast T_{\underline{X}_\Sigma}\simeq T_{\underline{X}_\Sigma}$. Those foliations are determined by a complex subspace $W\subseteq N\otimes \C$. More precisely, this is the tangent space of the foliation at the identity element $1\in T$. If $X_\Sigma$ is the corresponding log space of this toric variety (see Example \ref{ex: toric_log_str}),  any toric foliation produces a smooth log foliation via the inclusion $\OO_{X_\Sigma}\otimes_\C W\subseteq \OO_{X_\Sigma}\otimes_\Z N \simeq T_{X_\Sigma}$.
 \end{example}
 
To produce a wide range of examples, we explain below how to glue usual foliations on smooth varieties to obtain foliations on normal crossings singular spaces.

\subsection{Pushout  foliations}

We now move on to viewing foliations on varieties with normal crossings through the lens of the notions explained in the previous sections.
We will begin by briefly discussing the gluing process for foliations.

The gluing of schemes along subspaces corresponds to the notion of pushout. Although arbitrary pushouts may not exist in the same category, we do have an existence theorem that applies to our setting. By \cite[Theorem 5.4 and 7.1]{ferrand2003conducteur}, given a family of smooth varieties $X_i$ of dimension $n$ with simple normal crossing divisors $D_i$ and a family of isomorphisms 
$$ \varphi_{ij}:D_{ij}\subseteq D_i\to D_{ji}\subseteq D_j$$
    satisfying $\varphi_{ji}=\varphi_{ij}^{-1}$ and the cocycle condition on the triple intersections, we can define the pushout scheme $\coprod_D X_i$, which is the variety corresponding to the  gluing of the $X_i$ along this data and will have normal crossings depending on the combinatorics of $\{\varphi_{ij}\}$. When no confusion arises, we will write $\underline{X}$ instead of $\coprod_D X_i$.

\begin{lemma}
    Let $\underline{X}$ be a variety with simple normal crossings with double locus $D$, and let $X_1,\dots,X_r$ be its irreducible components. Then
    \begin{enumerate}
    \item The restriction map $T_{\underline{X}}\to T_{X_i}$
        factors through $T_{X_i}(-\log  (D\cap X_i))$, and
        \item the composition $T_{\underline{X}}\to T_{X_i}(-\log(D\cap X_i))\to i_*T_D$ induces an isomorphism $T_{\underline{X}} \simeq \bigtimes_{T_D}T_{X_i}(-\log D)=\{(v_i)\,\vert\, v_i\vert_D=v_j \vert_D\}$.
    \end{enumerate}
\end{lemma}
\begin{proof}
This follows straightforward from the computations made in Example \ref{ex:tangent_underline}.  
\end{proof}

Let $\F$ be a codimension 1 foliation on $X$ defined by a holomorphic $1$-form $\omega \in H^0(X,\Omega^1_X\otimes N_\F)$ and $Y$ a smooth $\F$-invariant hypersurface. Since the restriction of $\omega$ to $Y$ vanishes on $T_Y$, it factors trough the conormal bundle of $Y$. Hence we obtain section
$$\omega|_Y \in H^0\big(Y,N_{Y/X}^\vee \otimes N_\F|_Y\big)$$
We define $Z(\F,Y)$ as the zero divisor of $\omega|_Y$ (see \cite[p. 10] {pereira2024closed}).

\begin{lemma}\label{T_F_loc_free}
 Let $X$ be as in the lemma above, and let $\F_i$ be a foliation on $X_i$. Then:
    \begin{enumerate}
        \item the subsheaf $$ T_\F= \{v\in T_{\underline{X}}\,\vert\, v\vert_{X_i}\in T_{\F_i} \}\subseteq T_{\underline{X}}$$
        defines a foliation on $X$.
        \item if $\dim(\underline{X})=r=2$ and $D$ is a leaf of $\F_1$ and $\F_2$, then $T_\F \subseteq T_{\underline{X}}$ restricts to $T_{\F_i}\subseteq T_{X_i}$ if and only if  $Z(\F_1,D)=Z(\F_2,D)$. In this case, this  induces an isomorphism $T_{\F_1}\vert_D\simeq T_{\F_2}\vert_D$ under which $T_\F\simeq T_{\F_1}\times_D T_{\F_2} $, the pushout of line bundles as in \cite{ferrand2003conducteur}.
\end{enumerate}
\end{lemma}
\begin{proof} The involutivity of the sheaf $T_\F$ follows directly from the description of $T_{\underline{X}}$ in the lemma above. In order to see that the quotient $T_{\underline{X}}/T_\F$ is torsion-less and coherent it suffices to observe that $T_\F$ is the kernel of the morphism of coherent sheaves
$$ T_{\underline{X}}\to \bigoplus_{i=1}^r T_{X_i}\to \bigoplus_{i=1}^r N_{\F_i},$$
where we are omitting the pushforwards via the embeddings $X_i\hookrightarrow {\underline{X}}$. Indeed, since the sheaf on the right side is torsion-less and coherent we can deduce that the same holds for $T_{\underline{X}}/T_\F$ and $T_\F$ respectively. 

For the second assertion, notice that the hypotheses imply that $D$ is a smooth curve. Around each point $p\in D\subseteq X_1$ there exist a choice of coordinates $(x,y)$ with $D=\{x=0\}$. Let $\omega=Adx +Bdy$ be a local $1$-form defining $\F$. Since $N_{D/X_1}$ is spanned by the class of $\frac{\partial}{\partial x}$, we see that, locally around $p$, $Z(\F_1,D)=div(A)$. On the other hand, $v_1=-A\frac{\partial}{\partial y}+B\frac{\partial}{\partial x}$ is a local generator of $T_{\F_1}$, and its restriction to $D$ is $-A\frac{\partial}{\partial y}$. This shows that if $v_2$ is a local generator of $T_{\F_2}$, then $v_2\vert_D=u v_1$ for some $u\in \OO_D^*$. Extending $u$ to some neighborhood of $p$ in $X_1$ we can construct a local generator $v=(uv_1,v_2)$ of $T_\F$ that restricts to local generators on each $X_i$. 

Conversely, if $T_\F$ restricts to each $T_{\F_i}$ then the image of the restriction $T_\F\to i_*T_D$ coincides with the corresponding $T_{\F_i}\to i_*T_D$, giving an isomorphism $T_{\F_1}\vert_D\simeq T_{\F_2}\vert_D$. This allows to promote these objects to a line bundle $T_{\F_1}\times_D T_{\F_2}$ and a global section $v'\in T_{\underline{X}}\otimes (T_{\F_1}\times_D T_{\F_2})^{\otimes -1}$ defining $\F$.
\end{proof}

\begin{definition}
    Let $X$ be a simple normal crossing variety, together with foliations $\F_i$ on its irreducible components. Their pushout will be the foliation $\F=\coprod_D \F_i$ constructed above. In the case where $r=2$ we will often use the notation $\F=\F_1\sqcup_D \F_2$. 
\end{definition}

 \begin{remark}
Let $\underline X=\{xyz=0\}\subset \mathbb A^3$, with irreducible components
$X_1=\{x=0\}$, $X_2=\{y=0\}$ and $X_3=\{z=0\}$. Set
$D_{ij}=X_i\cap X_j$ and $D_{123}=X_1\cap X_2\cap X_3=\{0\}$.
On the components, consider the rank-one subsheaves generated by
$
e_1=y\partial_y-z\partial_z
$
on $X_1$,
$
e_2=x\partial_x-z\partial_z
$
on $X_2$, and
$
e_3=x\partial_x+\lambda y\partial_y
$
on $X_3$, where $\lambda\in\mathbb C^*$ and $\lambda\neq 1$. Thus
$T_{\F_i}=\mathcal O_{X_i}e_i$.
Along $D_{12}$, the restrictions of $e_1$ and $e_2$ are both
$-z\partial_z$. Along $D_{23}$, the restrictions of $e_2$ and $e_3$ are both
$x\partial_x$. Along $D_{31}$, one has
$e_1|_{D_{31}}=y\partial_y$ and
$e_3|_{D_{31}}=\lambda y\partial_y$. Hence the pairwise identifications are
given by
$\phi_{12}(e_1)=e_2$, $\phi_{23}(e_2)=e_3$ and
$\phi_{31}(e_3)=\lambda e_1$.
We claim that the sheaf $T_\F$ is not locally free rank-one sheaf on $\underline X$:
if it did, there would exist a local generator $v$ of $T_{\F}$.
On each component $X_i$, the restriction of $v$ must generate
$T_{\F_i}=\mathcal O_{X_i}e_i$. Hence
$v|_{X_i}=u_i e_i$
for some unit $u_i\in\mathcal O_{X_i,0}^{\times}$.
Since $v$ is a section of $T_{\underline X}$, the three vector fields
$u_i e_i$ must agree on the double strata, i.e., the equations $
u_1|_{D_{12}}=u_2|_{D_{12}}$, 
$u_2|_{D_{23}}=u_3|_{D_{23}}$, 
and 
$u_1|_{D_{31}}=\lambda u_3|_{D_{31}}$ must hold.
Evaluating at the triple point, we obtain
\[
u_1(0)=u_2(0)=u_3(0),
\qquad
u_1(0)=\lambda u_3(0).
\]
Since \(u_3(0)\neq 0\), this forces \(\lambda=1\), a contradiction.
Therefore the corresponding rank-one subsheaf is not locally free at
\(D_{123}\).
Thus, for
$r\geq 3$, compatibility along the double strata must also satisfy the
cocycle condition on the triple strata.
\end{remark}

\subsection{Foliated d-semistability} We can now move on to discussing foliated $d$-semistability.

\begin{definition}\label{def:semiestable_fol}
    Suppose $\F$ is a foliation on a normal crossing complex space $\underline{X}$. A \textit{logarithmic structure of semistable type for $(\underline{X},\F)$} is a logarithmic structure of semistable type $\M_X$ on $\underline{X}$ together with a semistable map  $(\underline{X},\M_X)\rightarrow (\ast,\N)$, such that $$T_\F\subseteq T_{(\underline{X},\M_X)/(\ast,\N)}.$$
We will say that a pair $(\underline{X},\F)$ is \textit{d-semistable} if there exists one such log structure, and $(X,\F)$ would be called a \textit{semistable foliated pair}. If $(\underline{X},\F)$ only admits a logarithmic structure of semistable type locally, we will say that this pair is \textit{locally d-semistable}. An \textit{isomorphism} between two logarithmic structures $\M_X$ and $\M'_X$ of semistable type for $(\underline{X},\F)$ is a commutative triangle 

\begin{center}
    \begin{tikzcd}[column sep=tiny] 
{(\underline{X},\M_X)} \arrow[rd] \arrow[rr,"\varphi"] &                & {(\underline{X},\M'_X)} \arrow[ld] \\
                                             & {(\ast,\N)} &                                   
\end{tikzcd}
\end{center}
where the horizontal map is an isomorphism of logarithmic structures.
\end{definition}

\begin{example}\label{ex:glueleaf} Let $X = X_1\sqcup X_2$ a simple normal crossing variety with double locus $D$. Assume there is a codimension one foliation $\F_i$ on $X_i$, smooth along $D$ and having such divisor as a leaf. We claim that the pushout foliation $\F =\F_1\sqcup_D \F_2$ on $X = X_1\sqcup_D X_2$ is locally d-semistable. Then we can find coordinates $(x_i, z_1,\cdots,z_{n-1})$ on $X_i$ such that $D_i = V(x_i)$ and the foliation $\F_i$ is given by the $1$-form $dx_i$. Then $\F$ is generated by the vector fields $z_0\partial_{z_0},\dots, z_{n-1}\partial_{z_{n-1}}$. 
\end{example}

\begin{example}\label{ex:transverse}
Let us now consider the opposite case, i.e. when two smooth foliations by curves $\F_1$ and $\F_2$ are transverse to the common divisor $D$. Similarly to the previous example, there must exist analytic coordinates 
$(x_i,z_1,\dots,z_{n-1})$ such that $D=V(x_i)\subseteq X_i$ and $T_{\F_i}=\langle \partial_{x_i}\rangle$. The foliation $\F_1\sqcup \F_2$ on
$X_1\sqcup_D X_2$ will be generated by the vector field 
$v=x_1\partial_{x_1}-x_2\partial_{x_2}$,
which is contained in $T_{X/(\ast, \N)}$ for the canonical log smooth morphism $X\to (\ast, \N)$.
\end{example}

\begin{example} \label{ex:babycamachosad} Consider the foliation $\F$ on $\underline{X}=\{xy=0\}\subseteq (\C^3,0)$ whose restriction to each irreducible component is induced by $v_1=\lambda_1y\partial_y+z\partial_z$ and $v_2=\lambda_2x\partial_x+z\partial_z$ respectively. Then $\F$ is semistable provided that $\lambda_1+\lambda_2=0$, since in this case the vector field $\lambda_1y\partial_y+\lambda_2 x\partial_x+z\partial_z\in T_{X/(\ast,\N)}$, where we are taking the standard log structure over the log point. The relation between local semistability and this residues will be further explored in Proposition \ref{prop:camacho-sad}.
\end{example}

The above example suggests that in the case where $D$ is $\F_i$-invariant, $d$-semistability imposes a strong local compatibility constraint between the foliations on each component:

\begin{proposition} \label{prop:camacho-sad}
Let $\F$ be a germ of semistable foliation on $\underline{X} = V(x_1\cdots x_r)\subseteq (\C^n,0)$, and denote by $\F_i$ the restriction of $\F$ to each irreducible component $X_i = V(x_i)$. If we assume that the divisors $D_i = D\cap X_i$ are $\F_i$-invariant, then their Camacho-Sad indexes satisfy the relations$$CS(\F_i,D_{ij},0) + CS(\F_j,D_{ij},0) = r-2.$$
\end{proposition}

\begin{proof}
By assumption, there exists a local semistable structure $(X,\F)\to(\ast,\N)$. Since $D$ is a leaf of $\F$, there is a surjection 
$$R:\bigwedge^{n-1}T_\F\to i_*\bigwedge^{n-1}T_D.$$
Let $v\in \bigwedge^{n-1}T_\F$ be any element such that $R(v)$ is nonzero along every component of $D$. The contraction of a top log-form $\Omega\in \Omega^n_{X/(\ast,\N)}$ yields a section $\omega\in \Omega^1_{X/(\ast,\N)}(U)$ for some $U$ containing all generic points of $D$ and such that $\omega(T_\F)=0$.

Writing $\omega = a_1 dx_1/x_1 + \cdots +a_r dx_r/x_r + \eta$, the restriction of the foliation $\F$ to $X_i$ is induced by  
\[
\omega_i = (a_1-a_i) \widehat{x}_{1i}dx_1 + \cdots + (a_r-a_i) \widehat{x}_{ri}dx_r + \widehat{x}_{i}\eta.
\]
Consider the algebraic leaf $D_{ij} = V(x_i,x_j)\subseteq X_i$. Then the Camacho-Sad index of $\F_i$ along $D_{ij}$ is
\begin{align*}
CS(\F_i, D_{ij}) &= \Res\bigg(\frac{1}{(a_j-a_i)\widehat{x}_{ij}}\Big(\sum_{k\neq i, j} (a_k-a_i) \widehat{x}_{ijk}dx_k + \widehat{x}_{ij}\eta\Big)\Big\vert_{D_{ij}}\bigg)\\
& = \Res\bigg(\frac{1}{(a_j-a_i)}\Big(\sum_{k\neq i, j} \frac{(a_k-a_i)}{x_k}dx_k  + \eta\Big)\Big\vert_{D_{ij}}\bigg)
\end{align*}
In particular, comparing these residues along different components we get
\begin{align*}
    CS(\F_i,D_{ij}) + CS(\F_j,D_{ij}) & = \Res\bigg(\frac{1}{(a_j-a_i)}\Big(\sum_{k\neq i, j} \frac{(a_k-a_i) + (a_j-a_k)}{x_k}dx_k\Big)\Big\vert_{D_{ij}}\bigg)\\
     & = \Res\bigg(\sum_{k\neq i, j} \frac{dx_k}{x_k}\Big\vert_{D_{ij}}\bigg) = r-2,
\end{align*}
finishing the proof.
\end{proof}

\begin{remark}\label{double_crossings}
  In the case where $\underline{X} = X_1\sqcup_D X_2$ has only double simple normal crossings, as we mentioned before, we have that \smash{$T^1_X\simeq I_{X_1}|_D\otimes I_{X_2}|_D\simeq N_{D/X_1}\otimes N_{D/X_2}$} and hence $\underline{X}$ will be d-semistable whenever \smash{$N_{D/X_1}\simeq N_{D/X_2}^{\vee}$}. This is compatible with the above formula, since by the Camacho-Sad Formula \cite{camacho1987pontos}, together with the above Proposition \ref{prop:camacho-sad} gives $\deg (N_{D/X_1})=-\deg (N_{D/X_2})$. 
\end{remark}

\subsection{Criterion for foliated semistability} For a general complex space $\underline{X}$ with cotangent complex $L_{\underline{X}}$, we consider for each integer $i\geqslant 0$ the coherent sheaf on $\underline{X}$ defined by
$$T_{\underline{X}}^i = \Ext^i(L_{\underline{X}},\OO_{\underline{X}}).$$
It is well-known that the sheaf of graded modules \smash{$T_{\underline{X}}^\bullet = \bigoplus_{i\geqslant 1} T_{\underline{X}}^i$} is endowed with a structure of graded Lie algebra  \cite{schlessinger1985lie}. Coming back to our case of interest, let $\underline{X}$ be a normal crossing variety. Then  $T_{\underline{X}}^i = 0$ for $i\geqslant 2$ because $\underline{X}$ is a local complete intersection. As a consequence, the graded Lie algebra structure on \smash{$T_{\underline{X}}^\bullet = T_{\underline{X}}^0\oplus T_{\underline{X}}^1$} is quite simple: besides the usual Lie concentrated in degree zero, it is completely determined by the bracket \smash{$[-,-]: T_{\underline{X}}^0\otimes T_{\underline{X}}^1 \rightarrow T_{\underline{X}}^1$}. 

Locally on $\underline{X}$, this operation becomes very explicit \cite[Section 4]{friedman1983global}. By the computations in \cite[Section 3.1]{sernesi2007deformations}, given a local embedding $\underline{X}= V(x_1\cdots x_r) \subseteq (\C^n,0)$,  we get an identification
\[
T^1_{\underline{X}} \simeq \OO_{\underline{X}}/(\hat{x}_1,\dots, \hat{x}_r)
\]
where \smash{$\hat{x}_i = x_1\cdots \hat{x}_i \cdots x_r$}. Then the bracket between the class of a function \smash{$g\in T_{\underline{X}}^1$} and a vector field \smash{$v = \sum_{i=1}^rb_ix_i\partial_{x_i}+\sum_{i>r} a_i \partial_{x_i}\in T_{\underline{X}}^0$} is 
\begin{equation} \label{eq: nabla_dual}
[v,g] =   v(g)-\bigg(\sum_{i=1}^rb_i\bigg) g.
\end{equation}
It is clearly $\OO_{\underline{X}}$-linear in the first variable and thus defines a \textit{connection} $\nabla$ on \smash{$T^1_{\underline{X}}$}.

\begin{lemma}\label{lema: connection}
Suppose $\F$ is a foliation on a normal crossing variety $\underline{X}$. The semistable log structure on $\underline{X}$ induced by a nowhere-vanishing global section \smash{$s\in H^0(\underline{X},T^1_{\underline{X}})$} is of semistable type for $(\underline{X},\F)$ if and only if  $s$ is \textit{$\F$-flat}, i.e. $\nabla_v(s)=0$ for every $v\in T_\F$.
\end{lemma}

\begin{proof}
It follows from Equation (\ref{eq: nabla_dual}) and Example \ref{ex:tangent_underline}.
\end{proof}

\begin{proposition}\label{prop: linear_hol}
    Let $\F=\F_1 \sqcup_D \F_2$ be a locally d-semistable foliation constructed as the pushout of two codimension one foliations along a common smooth leaf $D$. Consider $h_i':\pi_1(D,p)\rightarrow \C^\ast$ the linear part of the holonomy of $D$ seen as a leaf of $\F_i$. Then $\F$ is d-semistable if and only if $h'_1 = (h'_2)^{-1}$.
\end{proposition}
\begin{proof}
By the Remark \ref{double_crossings} we know that \smash{$T_{\underline{X}}^1\big\vert_D \simeq N_{D/X_1}\otimes N_{D/X_2}$}. Since $D$ is a smooth leaf of $\F_i$ we know that $T_{\F_i}|_D \simeq T_D$, and therefore restriction of the partial connection \smash{$\nabla: T_{\F}\otimes T_{\underline{X}}^1\rightarrow T_{\underline{X}}^1$} gives a connection $\nabla|_D$ on the line bundle \smash{$N_{D/X_1}\otimes N_{D/X_2}$}. The monodromy of this connections is precisely the product of the aforementioned linear holonomies $h'_1$ and $h'_2$, hence $\nabla|_D$ is flat if and only if $h'_1 = (h'_2)^{-1}$. Since a connection on $D$ is flat if and only if admits a nowhere-vanishing global section, the result follows from Lemma \ref{lema: connection}.
\end{proof}

\begin{remark} \label{remark: tangent and T1}
If $\F$ is a foliation on $\underline{X}$, and $X\to (\ast,\N)$ is a semistable log structure for $\F$, then the corresponding (relative) tangent sheaves fit in the following commutative diagram
\begin{center}
\begin{tikzcd}
            &                                               & 0 \arrow[d]                                    & 0 \arrow[d]                      &   \\
0 \arrow[r] & T_\F \arrow[r] \arrow[d, no head, double] & {T_{X/(\ast,\N)}} \arrow[r] \arrow[d]       & {N_{\F}} \arrow[r] \arrow[d] & 0 \\
0 \arrow[r] & T_\F \arrow[r]                                & T_{\underline{X}} \arrow[r] \arrow[d]          & N_{\underline{\F}} \arrow[r] \arrow[d]         & 0 \\
            &                                               & T^1_X \arrow[d] \arrow[r, no head, double] & T^1_X \arrow[d]                  &   \\
            &                                               & 0                                              & 0,                                &  
\end{tikzcd}
\end{center}
Observe that we are using a different notation for the usual normal sheaf and the one arising in the setting. 
\end{remark} 

\subsection{From local to global foliated d-semistability} Since d-semistability is a property of global nature, the local existence of semistable structures for foliations, as exemplified earlier in this section, is not enough to guarantee it. Nevertheless, in this scenario we should be able to compute the obstruction for the gluing of these log structures. In order to do so, we will follow the ideas of \cite{olsson2003universal} closely.

Given a normal crossing variety $X$ with singular divisor $D$, consider the sheaf of abelian groups $G=\ker(\OO^\ast_{\underline{X}}\to \OO^*_D)$ and its $\F$-invariant part 
$$G_\F :=\big\{u\in G:\;v(u) = 0\;\;\forall v\in T_\F\big\},$$
which is of course a subsheaf of the sheaf 
 $$\OO_{X/\F}^\ast : = \{u\in \OO_{\underline{X}}^\ast\; : \; v(u) = 0\;\forall\, v\in T_\F\}$$
 of invertible first integrals for $\F$.
 
By the proof of \cite[Theorem 3.14]{olsson2003universal} we know that the sheaf $\Aut(\M_X)$ of automorphisms of $\M_X$ is naturally isomorphic to $G$. Let us now show that under this isomorphism, the sheaf $\Aut_\F(\M_X)$ maps to $G_\F$. This correspondence can be described as follows:
any $\varphi\in\Aut(\M_U)$ is completely determined by some $u_i\in \OO_U^*$ such that $\varphi^\flat(e_i)=u_i e_i$\footnote{$\varphi^\flat(e_i)=\alpha^{-1}(u_i) e_i$}. The compatibility with respect to $\alpha:\M_U\to \OO_X$ forces $u_ix_i=x_i,$ or equivalently $(1-u_i)x_i=0$. By \cite[Lemma 3.15]{olsson2003universal}, this is equivalent to $u_i$ taking the form $u_i=1+a_ix_1\cdots x_{i-1}x_{i+1}\cdots x_r$. If we denote the product of these units by $u=u_1\cdots u_n$, the above isomorphism is
\begin{align*}
     Aut(\M_U)&\to G(U)\\
     \varphi &\mapsto u= 1+\sum a_ix_1\cdots x_{i-1}x_{i+1}\cdots x_r.
 \end{align*}
An element $\varphi\in G$ will be compatible with $\F$ exactly when 
$$ (v,\delta)(\varphi^b(e_1+\dots+e_r))=0,$$
for every $(v,\delta)\in T_\F$. Since $\varphi^b(e_1+\cdots+e_n)=u(e_1+\cdots+e_n)$, this is equivalent to $\delta(u)= v(u)u^{-1}=0,$ i.e. $u\in G_\F$.

In order to compute the desired obstruction, we will make use of the following fact.

\begin{proposition}\label{global_obs_semistability}
Let $\F$ be a foliation on a complex space $\underline{X}$ with normal crossings. If the pair $(\underline{X},\F)$ is locally d-semistable, there is a canonical obstruction in $$H^1\big(\underline{X},\OO^*_{\underline{X}/\F}/G_\F\big)$$ that vanishes if and only if $(\underline{X},\F)$ is d-semistable.
\end{proposition}
\begin{proof}
Consider the presheaf $SS_{\underline{X}}$ on $\underline{X}$ defined by
$$SS_{\underline{X}}(\underline{U}) = \big\{\text{semistable maps }f:(\underline{U},\M_{U})\rightarrow (\ast,\N)\big\}/\simeq$$
where we consider 
$$f_1:(\underline{U},\M^1)\to (\ast,\N) \simeq f_2:(\underline{U},\M^2)\to (\ast,\N)$$ to be if there exists an isomorphism of log spaces commuting with the projection to $(\ast, \N)$.
From \cite[page 424]{olsson2003universal} we know that this is in fact a sheaf. If we define $SS_{(\underline{X},\F)}$ as the subsheaf of $SS_{\underline{X}}$ of isomorphism classes of those log structures of semistable type for $(\underline{X},\F)$, our aim is to prove that $SS_{(\underline{X},\F)}$ is a $(\OO_{X/\F}^\ast/G_\F)$-torsor. First of all, $SS_{(\underline{X},\F)}(\underline{U})$ is non-empty for $\underline{U}\subseteq \underline{X}$ small enough by the local d-semistability of $(\underline{X},\F)$. The sheaf of abelian groups $\OO_{X}^\ast$ acts on $SS_{\underline{X}}$ via multiplication by units: given a log smooth morphism $f:(\underline{U},\M_U)\rightarrow (\ast,\N)$ and an element $u\in \OO_X^\ast(\underline{U})$, the map $u\cdot f:(\underline{U},\M_U)\rightarrow (\ast,\N)$ is determined by
$$(u\cdot f)^\flat: f^*\M_{(\ast, \N)}\xrightarrow{\cdot u} f^*\M_{(\ast, \N)} \xrightarrow{f^\flat} \M_X.$$
The subsheaf $SS_{(\underline{X},\F)}$ is stable under multiplication by $\OO_{X/\F}^\ast$ because for each $(v,\delta)\in T_\F$ and $f\in SS_{(\underline{X},\F)}$ as above,
\begin{align*}
    0 = \delta\big((u\cdot f)^\flat(1)\big) \iff 0 = \delta(u) + \delta\big(f^\flat(1)\big) = \delta(u) \iff v(u) = 0.
\end{align*}
By \cite[Proposition 3.8]{olsson2003universal}, two log structures of semistable type are locally isomorphic (as defined in Definition \ref{def:morphism}). This implies that the action is transitive. On the other hand we have $u\cdot f = f\in SS_{\underline{X}}$ exactly when $u\in G$, and hence the stabilizer is $G_\F$. 
\end{proof}

\begin{remark} \label{remark:restriction}
The kernel of the restriction map $\OO^\ast_{\underline{X}/\F}\rightarrow \OO_{D}^\ast$ is precisely $G_\F$. With this in mind, the obstruction above maps to the one constructed in \cite{olsson2003universal} via 
$$H^1\big(\underline{X},\OO^\ast_{\underline{X}/\F}/G_\F\big)\rightarrow  H^1\big(\underline{X},\OO^\ast_{D}\big),$$
indicating that of course if $(\underline{X},\F)$ is d-semistable then so is $\underline{X}$.
\end{remark}

This obstructions are easier to compute when $\underline{X}$ has simple normal crossings and $\F$ is induced by the gluing of well-behaved foliations on its irreducible components.
Let us now see how these groups look like in a handful of examples.

\begin{example} 
Suppose we the gluing of two foliations $\F =\F_1 \sqcup \F_2$ on a simple normal crossing variety $\underline{X}=X_1\sqcup_D X_2$ such that $\F_i=\F\vert_{U_i}$ is regular and transverse to $D\subseteq X_i$, where $U_i$ is a small neighborhood of $D\subseteq X_i$.
By the transversality condition, every element restricting to $1$ on $D$ and being constant along leaves (i.e., in $G_{\F}$) will be trivial. By the previous remark, this means that $\OO_{\underline{X}/\F}^\ast\hookrightarrow \OO_D^\ast$ under the restriction map.
On the other hand, in each component we can find coordinates $(x,z)$ such that $\F$ is generated by $v=\frac{\partial}{\partial x}$ and $D=\{x=0\}$, showing that $\OO_{X_i/\F_i}\to \OO_D$ is surjective. Therefore the obstruction for the $d$-semistability of $\F$ lies on $\Pic(D) = H^1(D,\OO_D^\ast)$. This is simple the class of the line bundle $T^1_{\underline{X}}\vert_D\simeq N_{D/X_1}\otimes N_{D/X_2}$ (see \cite{kato1996dlog_def}).
\end{example}

\begin{example}\label{ex:d_semi_leaf} 
Let $\F$ be a codimension $1$ foliation on a normal crossing variety $\underline{X}$ having $D$ as a leaf (this is, $D$ is invariant by $\F$ and $T_\F\vert_D\to T_D$ is generically an isomorphism). In particular, every first integral must be locally constant along $D$, and hence the restriction morphism in Remark \ref{remark:restriction} factors through the constant sheaf: 
$$\OO^\ast_{X/\F}/G_\F\rightarrow \underline{\C}^\ast_D\subseteq \OO_{D}^\ast.$$
Since this is clearly surjective, we can conclude $\OO^\ast_{X/\F}/G_\F\simeq \underline{\C}^\ast_D$. In particular we see that in this case the obstruction above is of topological nature. Is $\F = \F_1\sqcup_D \F_2$ the gluing of two foliations along a smooth leaf, then then obstruction on $H^1(D,\C^\ast)$ for the $d$-semistability of $\F$ is precisely the quotient of the linear of holonomies of each $\F_i$ as in Proposition \ref{prop: linear_hol}. This gives plenty of $d$-semistable foliated varieties. For instance if $(\underline{X},\F)$ corresponds to the gluing of two fiber bundles along a smooth fiber, then by Example \ref{ex:glueleaf} there will exist a semistable log-structure compatible with $\F$.
\end{example}

\section{Log deformations}
Inspired by \cite{gomez1988transverse} and \cite{kato1996dlog_def}, we will now develop an infinitesimal deformation theory for log smooth foliated morphisms. 

Let $Q$ be a finitely generated integral saturated monoid whose only invertible element is $1$, and $(\ast,Q)$ the corresponding logarithmic point over $\C$, together with a log-smooth morphism $X\to (\ast,Q)$, and a foliation $\F$ on $X$ such that $T_\F\hookrightarrow T_{\underline{X}}$ factors through $T_{X/(\ast,Q)}\to T_{\underline{X}}$ (i.e., the foliation and the log smooth map are \emph{compatible}). Unless said otherwise, this data will be fixed for rest of the section.

We will consider the category $\mathcal A_Q$ of Artinian local $\C[[Q]]$-algebras with residue field $\C$, where $\C[[Q]]$ is the completion of $\C[Q]$ along the maximal ideal $\C[Q\setminus\{1\}]$.
Observe that for any object $A\in \mathcal A_Q$, the space $\Spec A$ has a natural logarithmic structure of the form $Q\oplus A^*\to A$ induced by the composition $Q\to \C[[Q]]\to A$. In order to have a cleaner notation we will often write  $(A,Q)$ for this logarithmic scheme.

\begin{definition}
    A \textit{logarithmic deformation} of $X$ over some $A\in \mathcal A_Q$ is the data of a scheme $X_A$ and a log smooth morphism $X_A\to (A,Q)$ fitting in a cartesian diagram 
    \begin{center}
        \begin{tikzcd}
            X \arrow[d]\arrow[r] & X_A \arrow[d]\\
        (\ast,Q) \arrow[r] & (A,Q).
        \end{tikzcd}
    \end{center}
    Two deformations are isomorphic if there exists an isomorphism of logarithmic schemes over $(\Spec A,Q)$ restricting to the identity on the central fiber. 
\end{definition}

It is worth pointing out that by \cite[Section 1.7]{kato1988logarithmic_str}, in the situation above we also have a canonical isomorphism $T_{X_A/(A,Q)}\vert_X\simeq T_{X/(\ast,Q)}$. Moreover, by \cite[Corollary 4.5]{kato1988logarithmic_str}
the underlying morphism of schemes $\underline{X}_A\to A$ is flat.

\begin{definition} Let $X\to(\ast,Q)$ be a log smooth morphism and $\F$ a foliation
    A \textit{logarithmic deformation} of $(X,\F)$ over $A\in \mathcal A_Q$ is a logarithmic deformation $X_A$ of $X$ together with a log foliation $\F_A$
        \[
0\rightarrow T_{\F_A}\rightarrow T_{X_A/(A,Q)}\rightarrow N_{\F_A}\rightarrow 0,
\]
    tangent to the fibers of $X_A\to (A,Q)$, restricting to $\F$ on $X$, and such that the normal sheaf $N_{\F_A}$ is torsion-less and flat over $A$. Two deformations $(X_A,\F_A)$ and $(\overline{X}_A,\overline{\F}_{A'})$ are isomorphic if there exists an isomorphism of deformations between $X_A$ and $\overline{X}_A$ sending $\F_A$ to $\overline{\F}_{A'}$ 
\end{definition}

\begin{definition}
    The \emph{log smooth deformation functor of} $(X,\F)\to (\ast,Q)$ is the contravariant functor $Def^{ls}_{(X,\F)}$ assigning to each $A\in \mathcal A_Q$ the set of equivalence classes of log-smooth deformations of $(X,\F)\to(\ast,Q)$.
\end{definition}

If $X_A\to (A,Q)$ is a logarithmic deformation of $X\to (\ast, Q)$ over some $A\in \mathcal A_Q$ and $A'\to A$ is a surjective arrow in $\mathcal A_Q$, then a lifting of $X_A$ is a deformation over $A'$ restricting to $X_A$ over $A$. Recall that an extension is \emph{small} if its kernel $J$ satisfies $J\cdot\mathfrak m_{A'}=0$, and that every extension in $\mathcal A_Q$ can be described in terms of these.
The following proposition is a good illustration of the local deformation theory of smooth logarithmic structures.

\begin{proposition}[{{\cite[Proposition 3.14]{kato1988logarithmic_str}}}] \label{prop:logdefX}
    Let $X_A\to (A,Q)$ be a logarithmic  deformation of $X\to (\ast, Q)$, and $0\rightarrow J \rightarrow A'\to A\rightarrow 0$ a small extension. Then the following holds:
    \begin{enumerate}
        \item If $X$ is Stein, a lifting $X_{A'}$ of $X_A$ over $A'$ exists and is unique up to isomorphism.
        \item The sheaf of automorphisms of the liftings $X_{A'}\to (A',Q)$ relative to $X_A\to (A,Q)$ is canonically isomorphic to $$H^0\big(X,T_{X/(\ast,Q)}\big)\otimes J.$$
        \item If non-empty, the set of liftings to $A'$ is in bijection with the vector space $$H^1\big(X,T_{X/(\ast,Q)}\big)\otimes J.$$
        \item The obstruction to finding a lifting to $A'$ lies in the group $$H^2\big(X,T_{X/(\ast,Q)}\big)\otimes J.$$
    \end{enumerate}
\end{proposition}

Although this statement follows from the general theory of torsors, with the exception of the first item as explained loc. cit., we will need to manipulate the technical caveats behind it in order to develop a treatment similar to \cite{gomez1988transverse}. Most of the computations carried out in that work translate directly to this context. One of our objectives for this section, namely the construction of an obstruction theory for our functor, requires to further expand these technicalities. It will be hence convenient to carefully adapt those ideas to our context before moving on.

The following lemma can be found in \cite[Section 3]{abramovich2010logarithmic}.

\begin{lemma}
    If $\imath: X\to \X$ is a thickening between fine log schemes then we have an exact sequence
    $$ 1 \rightarrow 1+I_{X/\X}\rightarrow \M_{\X} \xrightarrow{\imath^b} \M_{X}\rightarrow 1.$$
    If moreover $f$ is a thickening of first order then the sheaf of monoids $1+I_{X/\X}$ is isomorphic to the additive monoid $I_{X/\X}$ via the map $1+i\mapsto i$.
\end{lemma}

In particular, the above holds in the context where $\X'=X_{A'}\to (A',Q)$ is log smooth and $\X=X_A$ is the pullback along a small extension 
$$0\rightarrow J\rightarrow A'\rightarrow A\rightarrow 0.$$
For every automorphism $\varphi$ of $X_{A'}$ over $ (A',Q)$ relative to $A$ induces a commutative diagram
    \begin{center}
        \begin{tikzcd}
           1\arrow[r]& J\cdot \OO_X\arrow[r]\arrow[d] &\M_{X_{A'}} \arrow[d,"\varphi^b"]\arrow[r,"\imath^b"] & \M_{\X_{A}} \arrow[d,equal]\arrow[r] & 1\\
        1\arrow[r] & J\cdot \OO_X \arrow[r] &\M_{X_{A'}} \arrow[r,"\imath^b"] & \M_{X_{A}} \arrow[r] & 1,
        \end{tikzcd}
    \end{center}
where the rows are two copies of the sequence appearing in the lemma above. Since $f$ restricts to the identity on $X_{A}$, at the scheme level it must be locally of the form $id_{X_{A'}}+v_i$ for some gluing family of vector fields $v_i\in T_X\otimes J$, by Theorem \ref{prop:logdefX}.
In a similar fashion, after a bit of diagram chasing we can conclude that $\varphi^b$ must be of the form $\varphi^b=id_{\M_{X_{A'}}}+\delta$ for a monoid morphism $\delta: \M_{X_{A'}}\to J\otimes\OO_X$, that in fact factors through the arrow $i_0^b:\M_{X_{A'}}\to \M_X$, where $i_0:X\to X_{A'}$ is the inclusion of the central fiber. Using the same symbol for the corresponding arrow $\M_X\to J\otimes \OO_X$, we get a pair $(v,\delta)\in T_X$ as in Definition \ref{def:logderivation}.

With this in mind, we can now expand on the third item of Proposition \ref{prop:logdefX}.  
Consider a small extension $$0\rightarrow J\rightarrow A'\rightarrow A\rightarrow 0$$ and two liftings $X_{A'}$ and $\overline{X}_{A'}$ of a deformation $X_{A}$ over $A$. Since by Proposition \ref{prop:logdefX} these extensions are locally isomorphic, there exists an affine cover $\{U_i\}$ of $\underline{X}$ and (non-canonical) isomorphisms $\varphi_i:X_{A'}|_{U_i}\rightarrow\overline{X}_{A'}'|_{U_i}$. The compositions 
$$\varphi_{ij} = \varphi^{-1}_j\varphi_i: X_{A'}|_{U_{ij}}\rightarrow X_{A'}|_{U_{ij}}$$ 
are automorphisms satisfying the cocycle conditions $\varphi_{ik}=\varphi_{jk}\varphi_{ij}$. From the preceding discussion we know that at the level of structure sheaves and logarithmic structures we have $\varphi_{ij}^\ast = 1 + v_{ij}$ and $\varphi_{ij}^\flat = 1 + \delta_{ij}$ where $\{\theta_{ij}\}=\{(v_{ij},\delta_{ij})\}\in C^1(X,T_X\otimes J)$. It is important to point out that different choices of these local isomoprhism give the same element in $H^1(X,T_X\otimes J)$.

At the level of vector fields, on the other hand, we see that the derivative of $id+\theta_{ij}$ is $1 + [\theta_{ij},-]$ where $[-,-]$ is the log Lie bracket:

\begin{lemma} \label{lemma:Liederivative} Let $X_{A'}$ be a deformation of $X$ over $A'$, $\{\theta_{ij}\}\in C^1(\underline{X},T_X\otimes J)$ and $\varphi=1+\theta_{ij}$ the automorphism of $X_{A'}$ defined above. Then its differential is the map $d\varphi:w\mapsto w+[\theta_{ij},w]$. 
\end{lemma}

\begin{proof} We can check the identity $d\varphi(w)=[\theta_{ij},w]$ at the schematic and monoidal level separately. If $w=(v,\delta)\in T_X$ and $f\in \OO_X$, the first part follows from the identity
\begin{align*}(\varphi_{ij}^\ast)^{-1}v(\varphi^\ast_{ij}(f)) &= (1-v_{ij})v((1+v_{ij})f)\\ 
&= v(f) + [v_{ij},v](f).
\end{align*}
For the second coordinate, observe that for $m\in \M_{X_{A'}}$ we have
\begin{align*}
(\varphi^\ast)^{-1}\delta(\varphi^{\flat}(m) )&= (1-v_{ij})\delta((1+\delta_{ij}(m))m) \\
&=(1-v_{ij}) (\delta(1+\delta_{ij}(m))+\delta(m)).
\end{align*}
From the definition of log-derivations and the fact that and $J^2=0$ we can deduce that \begin{align*}
    \delta(1+\delta_{ij}(m))&=v(1+\delta_{ij}(m))(1+\delta_{ij}(m))^{-1}\\
    &=v(\delta_{ij}(m))(1-\delta_{ij}(m))\\
    &=v(\delta_{ij}(m)),
\end{align*} 
and hence 
$(\varphi^\ast)^{-1}\delta(\varphi^{\flat}(m) )=\delta(m)+v(\delta_{ij}(m))-v_{ij}\delta(m)$, finishing the proof.
\end{proof}

Now let $(X_{A'},\F_{A'})\to (A',Q)$ be a logarithmic deformation of $(X,\F)$. The restriction of this deformation to $A$ yields a diagram of the form

\begin{center}
\begin{tikzcd}
            & 0 \arrow[d]                         & 0 \arrow[d]                        & 0 \arrow[d]                         &   \\
0 \arrow[r] & T_{\F}\otimes J \arrow[r] \arrow[d] & T_{X}\otimes J \arrow[r] \arrow[d] & N_{\F}\otimes J \arrow[r] \arrow[d] & 0 \\
0 \arrow[r] & T_{\F_{A'}} \arrow[r] \arrow[d]     & T_{X_{A'}/(A',Q)} \arrow[r] \arrow[d] & N_{\F_{A'}} \arrow[r] \arrow[d]     & 0 \\
0 \arrow[r] & T_{\F_{A}} \arrow[r] \arrow[d]      & T_{X_A/(A,Q)} \arrow[r] \arrow[d]   & N_{\F_A} \arrow[r] \arrow[d]        & 0 \\
            & 0                                   & 0                                  & 0                                   &  
\end{tikzcd}
\end{center}

where the columns are obtained by tensoring by $\otimes_{\OO_{X_{A'}}}T_{X_{A'}/(A',Q)}$ the corresponding short exact sequences.

If $(\overline{X}_{A'},\overline{\F}_{A'})\to (A',Q)$ is another deformation coinciding with $(X_{A'},\F_{A'})$ over $A$, then as established above there exists an open cover $\{U_i\}$ and isomorphisms $\varphi_i:X_{A'}\vert_{U_i}\to \overline{X}_{A'}\vert_{U_i}$ over $A'$. In these opens sets, the data of the foliations is given by injections \smash{$f_i^{A'}:T_{\F_{A'}}\rightarrow T_{X'_{A'}/(A',Q)}$} and \smash{$\overline{f}_i^{A'}: T_{\overline{\F}'}\rightarrow T_{\overline{X}_{A'}/(A',Q)}$} extending $T_{\F_A}\hookrightarrow T_{X_{A}/(A,Q)}$. In particular, on $X_{A'}\vert_{U_i}$ both $T_{\F_{A'}}$ and
$\varphi_i^*T_{\overline{\F}_{A'}}$
are extensions of $T_{\F_A}$.

If $T_{\F}$ is locally free\footnote{The hypothesis on $T_{\F}$ being locally free is not indispensable at this point but reduces by a great amount the technicalities needed for part of this section - just as in \cite{gomez1988transverse}. In order to work without this assumption, one should proceed as in the proof of \cite[Lemma 2.2.6]{huybrechts2010geometry}.}, we can pick local isomorphisms $\psi_i :(\varphi_i)_\ast T_{\F_{A'}}|_{U_i} \rightarrow T_{\overline{\F}_{A'}}|_{U_i}$ and consider the compositions
\begin{equation}
(\varphi_j)_\ast T_{\F_{A'}}|_{U_{ij}} \xrightarrow{\psi_j} T_{\overline{\F}_{A'}}|_{U_{ij}} \xrightarrow{\psi_i^{-1}} (\varphi_i)_\ast T_{\F_{A'}}|_{U_{ij}}.
\end{equation}

Composing further with $(\varphi_j^{-1})_\ast$ we get   $\psi_{ij}:T_{\F_{A'}}|_{U_{ij}} \rightarrow (\varphi_{ij})_\ast T_{\F_{A'}}|_{U_{ij}}$ which is an isomorphism of $\OO_{X_{A'}}$-modules.
Equivalently, we get $A'$-linear maps of the form  $1+\beta_{ji}:T_{\F_{A'}}|_{U_{ij}} \rightarrow T_{\F_{A'}}|_{U_{ij}}$ for some  $\beta_{ji}:T_{\F_{A'}}|_{U_{ij}} \rightarrow T_{\F_{A'}}|_{U_{ij}}\otimes J$ satisfying the equation
\[
(1+\beta_{ji})(fv) = \varphi_{ij}^\ast(f) (1+ \beta_{ji})(v),
\]
which reduces to
\begin{equation*}
\beta_{ji}(fv) = f\beta_{ji}(v) + v_{ij}(f)v.
\end{equation*}
By construction, these elements satisfy the equation $(1+\beta_{ji})(1+\beta_{kj})=(1+\beta_{ki})$, which is of course equivalent to
$$ \beta_{ji}+\beta_{kj}=\beta_{ki}. $$

The identifications we made in the paragraph above also allow us to translate the gluing data corresponding to $T_{\overline{\F}_{A'}}\hookrightarrow T_{\overline{X}_{A'}/(A',Q)}$ in terms of the the cocycles defining the deformation $(X_{A'},\F_{A'})$.
Indeed, appropriately composing these inclusions with the identifications $\psi_i$ yields a map
$$(\varphi_i)_\ast T_{\F'}|_{U_i}\xrightarrow{\psi_i}  T_{\overline{\F}'}|_{U_i}\hookrightarrow T_{\overline{X}_{A'}/(A',Q)}|_{U_i}\xrightarrow{d\varphi_i^{-1}} (\varphi^{-1}_i)^\ast T_{X_{A'}/( A',Q)}|_{U_i}.$$
By adjunction, this gives a morphism $T_{\F_{A'}}\vert_{U_i}\to T_{X_{A'}/(A',Q)}\vert_{U_i}$ that coincides with $f_i^{A'}$ over $A$ and is therefore of the form  $f_i^{A'}+g_i:T_{\F_{A'}}|_{U_i}\rightarrow T_{X_{A'}/(A',Q)}|_{U_i}$, for some choice of $\{g_i\}\in C^0(X,Hom_{\OO_X}(T_{\F}, T_X)\otimes J)$.

\begin{lemma}
    Let $0\rightarrow J\rightarrow A'\rightarrow A\rightarrow 0$ be a small extension and $(X_{A'},\F_{A'})\to (A',Q)$ and $(\overline{X}_{A'},\overline{\F}_{A'})\to (A',Q)$ be two extensions of $(X_A,\F_A)$. Let $\theta_{ij}$, $f_i$, $\overline{f}_i$, $\beta_{ji}$ and $g_i$ be the elements defined above. Then
    $$ g_j-g_i= [\theta_{ij},f_j]+f_i\beta_{ij},$$
where $f_j: T_{\F}\vert_{U_j}\hookrightarrow T_{X/(\ast,Q)}\vert_{U_j}$ is the inclusion defining $\F$.
\end{lemma}

\begin{proof}
By pulling back the gluing data of $(\overline{X},\overline{\F})$ and using  Lemma \ref{lemma:Liederivative} we get a commutative diagram
   
    \begin{center}\Large
        \begin{tikzcd}
            T_{\F_{A'}}\vert_{U_{ij}}\arrow[d,"f_j+g_j"']\arrow[r, "1+\beta_{ij}"] & T_{\F_{A'}}\vert_{U_{ij}} \arrow[d,"f_i+g_i"]\\T_{X_{A'}/{(A',Q)}}|_{U_{ij}}  & T_{X_{A'}/{(A',Q)}}|_{U_{ij}} \arrow[l, "{1+[\theta_{ij},-]}"].
        \end{tikzcd}
    \end{center}
Its commutativity of the diagram is equivalent to the relation 
$$ (1+[\theta_{ij},-])(f_i+g_i)(1+\beta_{ij})=f_j+g_j,$$
which together with the fact that $J\cdot \m_A=0$ easily implies the equation in the Lemma.
\end{proof}

\begin{proposition} \label{prop:logdef}
    Let $0\rightarrow J\rightarrow A'\rightarrow A\rightarrow 0$ be a small extension and $(X_{A'},\F_{A'})\to (A',Q)$ be a (possibly non-involutive) deformation of $(X,\F)$. Then, the set of equivalence classes of deformations of $(X,\F)$ over $A'$ coinciding with $(X_{A'},\F_{A'})$ over $A$ is in bijection with the set
    $$ \Big\{(\{\theta_{ij}\},\{\overline{g_i}\})\in C^1(X,T_X\otimes J)\times C^0 (X,Hom(T_\F,N_\F)\otimes J)\,\Big\vert \,\overline{g_j}-\overline{g_i}= \overline{\, [\theta_{ij},f_j]}\Big\}\big/ \sim ,$$
    where $(\{\theta_{ij}\},\{\overline{g_i}\})\sim (\{\theta_{ij}'\},\{\overline{g_i}'\})$ if there exists some $\{w_i\}\in C^0(X,T_X)$ such that 
    $$ \theta_{ij}-\theta_{ij}'=w_i-w_j \,\, \mbox{ and }\,\, \overline{g}_i-\overline{g}_i'=\overline{\, [w_{i},f_i]}.$$
    Moreover, if $\F$ is involutive then the equivalence classes of involutive deformations correspond to the classes $(\{\theta_{ij}\},\{\overline{g}_i\})$ such that 
    $$ \overline{g}_i([v_1,v_2])-[v_1,\overline{g}_i(v_2)]+[v_2,\overline{g}_i(v_1)]=0$$
    for every pair of local sections $v_1,v_2\in T_\F$.
\end{proposition}
\begin{proof}
   With what we have done so far, this follows mutatis mutandis from the proof of \cite[Theorem 1.4]{gomez1988transverse} and \cite[Theorem 2.2]{gomez1988transverse}. 
\end{proof}

\begin{definition}
    Let $\underline{X}$ be an affine scheme over $S$. A sheaf (or a family) of Lie algebras on $\underline{X}$ is a locally free $\OO_{\underline{X}}$-module $\mathcal G $ equipped with a $\OO_S$-linear Lie bracket satisfying the Jacobi relations. A subalgebra of $\mathcal G$ is just a locally free subsheaf closed under the Lie bracket.
\end{definition}

\begin{definition}
    Let $0\to J\to A'\to A\to 0$ be a small extension and $\underline{X}_{A'}$ be a deformation of $\underline{X}$ over $A'$. A deformation of some sheaf of algebras $\mathcal G$  along $X_{A'}$ is a family of algebras $\mathcal{G}'$ on $\underline{X}_{A'}$ restricting to $\mathcal G$ on the central fiber.
\end{definition}

Let us now first compute the obstructions to \emph{locally} extending a deformation of foliation. The  following Lemma is known for the experts, but have included a proof since we did not find a reference where the obstructions are explicitly computed. For a sheaf of subalgebras $\mathcal H\subseteq \mathcal G$ on some affine scheme $X$, we will consider the Chevalley-Eilenberg complex $ (\sHom_{\OO_X}(\bigwedge^\bullet \mathcal H,\mathcal G/\mathcal H),\delta)$ and its associated cohomology groups.

\begin{lemma}
\label{lema:defLie}    Let $\mathcal H$ be a subalgebra of a Lie algebra $\mathcal G$ and $0\to J\to A'\to A\to 0$ be a small extension. If $\mathcal{H}_A\subseteq \mathcal G_A$ is a deformation of $\mathcal H\subseteq \mathcal G$ and $\mathcal G_{A'}$ is an extension of $\mathcal G_A$ over $A'$, then the obstruction to extend $\mathcal H_A$ to some $\mathcal H_{A'}\subseteq \mathcal G_{A'}$ lies in the group $H^2(\mathcal H,\mathcal G/\mathcal H)\otimes J$, where $H^2$ stands for the second Chevalley-Eilenberg cohomology of the $\mathcal H$-module $\mathcal{G}/\mathcal H$.
\end{lemma}
\begin{proof}
Let us first observe that since $\mathcal H$ is locally free,  so is $\mathcal H_{A}$. This means that we can locally extend the morphism 
of $\OO_{X_A}$-modules $f_A:\mathcal{H}_A\hookrightarrow \mathcal{G}_A$ to some $\OO_{X_{A'}}$-linear arrow $f_{A'}:\mathcal{H}_{A'}\hookrightarrow \mathcal{G}_{A'}$.
In the same manner, we can also extend the Lie bracket
$$ [\,,]_A: \bigwedge^2\mathcal H_{A}\to \mathcal H_{A}$$
to some $\OO_{X_{A'}}$-linear
$$ \{\,,\}_{A'}: \bigwedge^2\mathcal H_{A'}\to \mathcal H_{A'}.$$
If we were too lucky, these extensions would satisfy the equation 
$$ f_{A'}\{x,y\}=[f_{A'}(x),f_{A'}(y)]_{A'},$$
providing an extension of $\mathcal H_{A}\hookrightarrow \mathcal G_A$ to $A'$. 
If this was not the case, we could aim to modify them to force this equality. We will now see that the obstruction to do so lies in the claimed group.

Let us consider the map $b:\bigwedge^2\mathcal H_{A'}\to J\otimes_{A'}(\mathcal G_{A'}/\mathcal H_{A'})\simeq J\otimes_\C \mathcal G/\mathcal H$
defined as 
$$ b(x,y)=\overline{[f_{A'}(x),f_{A'}(y)]_{A'}-f_{A'}(\{x,y\}_{A'})}.$$
Since $m_{A'}\cdot J=0$, $b$ factors through some $\overline{b}:\bigwedge^2\mathcal H\to  J\otimes_\C \mathcal G/\mathcal H$:

there are natural isomorphisms 
$$\Hom_{A'}\bigg(\bigwedge^k \mathcal H_{A'}, J\otimes_{A'} \mathcal G_{A'}/\mathcal H_{A'}\bigg)\simeq \Hom_{\C}\bigg(\bigwedge^k\mathcal H,J\otimes_\C \mathcal G/\mathcal H\bigg). $$
If $v_1,v_2,v_3\in \mathcal H$ then
\begin{align*}
    \delta\overline{b}(v_1,v_2,v_3)&= \overline{\sum_{1\leq i\leq 3} (-1)^{i+1}[f(v_i),\overline{b}(\widehat{v}_i)]+\sum_{i<j} (-1)^{i+j}\overline{b}([v_i,v_j],v_k)}\\
    &=\overline{[f(v_1),\overline{b}(v_2,v_3)]}-\overline{[f(v_2),\overline{b}(v_1,v_3)]}+ \overline{[f(v_3),\overline{b}(v_1,v_2)]}\\ &\quad\quad- \overline{b}([v_1,v_2],v_3)+ \overline{b}([v_1,v_3],v_2) - \overline{b}([v_2,v_3],v_1).
\end{align*}
In order to compute this, we can pick liftings $v'_i\in \mathcal H_{A'}$ of $v_i\in \mathcal H=\mathcal H_{A'}/\mathfrak m_{A'}\mathcal H_{A'}$, so that
\begin{align*}
\big[f(v_i),\overline{b}(v_j,v_k)\big]&=\overline{\big[f_{A'}(v'_i),[f_{A'}(v'_j),f_{A'}(v_k)]_{A'}-f_{A'}(\{v'_j,v'_k\}_{A'})\big]}_{A'}.
\end{align*}
Putting the last two expressions together and using the fact that $[ \,,]_{A'}$ satisfies de Jacobi identity gives $\delta \overline{b}=0$.

If we had chosen a different lifting for the bracket, say $\{\,,\}'_{A'}=\{\,,\}_{A'}+\eta$ for some $\eta:\bigwedge^2\mathcal H_{A'}\to J\otimes_{A'}\mathcal H_{A'}$, then 
\begin{align*}
[f_{A'}(x),f_{A'}(y)]_{A'} - f_{A'}\{x,y\}
&=[f_{A'}(x),f_{A'}(y)]_{A'}-f_{A'}(\{x,y\}'_{A'}) \\
&=[f_{A'}(x),f_{A'}(y)]_{A'}-f_{A'}(\{x,y\}_{A'})-f_{A'}(\eta(x,y)) \\
&=b(x,y)-f_{A'}(\eta(x,y)),
\end{align*}
which indeed coincides with $b$ after passing to the quotient $\mathcal G_{A'}/\mathcal H_{A'}$.

On the other hand, a different choice of $f_{A'}$ changes $\overline{b}$ by a Chevalley--Eilenberg boundary: for a morphism of the form $f'_{A'}=f_{A'}+h$, with $h:\mathcal H_{A'}\to J\otimes\mathcal G_{A'}$, we get
\begin{align*}
b'(x,y)
&=
[f'_{A'}(x),f'_{A'}(y)]_{A'}-f'_{A'}(\{x,y\}_{A'}) \\
&=
[f_{A'}(x)+h(x),\,f_{A'}(y)+h(y)]_{A'}
   -f_{A'}(\{x,y\}_{A'})-h(\{x,y\}_{A'}) \\
&=b(x,y)+[f_{A'}(x),h(y)]_{A'}-[f_{A'}(y),h(x)]_{A'}
   -h(\{x,y\}_{A'}).
\end{align*}
On $J\otimes \mathcal G_{A'}/\mathcal H_{A'}\simeq J\otimes_\C(\mathcal G/\mathcal H)$, this becomes
\[
\bar b'(x,y)
=
\bar b(x,y)
+\overline{[f(x),h(y)]}
-\overline{[f(y),h(x)]}
-\overline{h([x,y])},
\]
because $\{x,y\}_{A'}$ reduces to $[x,y]\in\mathcal H$.
Observe that the right-hand side is
$\bar b'(x,y)=\bar b(x,y)+(\delta\bar h)(x,y)$,
where $\bar h:\mathcal H\to J\otimes_\C(\mathcal G/\mathcal H)$ denotes the class of $h$. This shows that the obstruction is a well-defined class in
$
H^2(\mathcal H,\mathcal G/\mathcal H)\otimes J.
$
Moreover, this obstruction is zero precisely when $\overline{b}=\delta \overline{h} $ for some $\overline{h}:\mathcal H\to J\otimes_\C(\mathcal G/\mathcal H)$. In this case, choosing an $\OO_{X_{A'}}$-lifting
$$
h:\mathcal H_{A'}\to J\otimes_{A'} \mathcal G_{A'},
$$
we see by the argument above that $f'_{A'}=f_{A'}-h$
is a $\OO_{X_{A'}}$-linear morphism that satisfies
\[
[ f'_{A'}(x),f'_{A'}(y)]_{A'} \in f'_{A'}(\mathcal H_{A'})
\]
for every $x,y$, hence defining an extension of $\mathcal H_{A'}\subseteq \mathcal G_{A'}$ of $\mathcal H_A\subseteq \mathcal G_A$.
\end{proof}

We will now define a logarithmic version of the leaf complex (see \cite{gomez1988transverse}). For an involutive subsheaf $T_\F$ of $T_{X/(\ast,Q)}$, the Lie bracket induces a logarithmic Bott connection on $N_\F$ - just as in the case of a subalgebra of a Lie algebra treated above. With this in mind, we can define the \emph{logarithmic leaf complex}
\[
\mathcal L^\bullet(\F):\quad T_X\rightarrow \sHom_{\OO_X}\big(T_\F,N_\F\big) \rightarrow \sHom_{\OO_X}\big(\wedge^2T_\F,N_\F\big) \rightarrow \cdots,
\]
where the differential is the Chevalley-Eilenberg differential
\begin{align*}
(\delta\varphi)(v_1,\dots,v_{k+1}) = &\sum_{1\leqslant i\leqslant k+1} (-1)^{i+1}\big[\overline{v}_i,\varphi(v_1,\dots,\widehat{v}_i,\dots,v_{k+1})\big]\\
&+\sum_{1\leqslant i<j\leqslant k+1} (-1)^{i+j} \varphi\big(\big[v_i,v_j\big],v_1,\dots,\widehat{v}_i,\dots,\widehat{v}_j,\dots,v_{k+1}\big).
\end{align*}

As promised, the hypercohomology of this complex will govern the log-smooth deformations of $(X,\F)\to (\ast,Q)$. Recall that, given an Stein cover $\U = \{U_i\}$ of $\underline{X}$, the hypercohomology groups of the leaf complex can be computed as the cohomology of the total complex of the double complex \smash{$\big(C^\bullet(\U, \mathcal L^\bullet(\F)), \delta, \hat\delta\big)$} where $\hat\delta$ is the $\check{C}$ech differential.

\begin{theorem} \label{teo:defobst}
Let $X\to(\ast,Q)$ be a log-smooth morphism and $\F$ a compatible foliation, and let $0\to J\to A'\to A\to 0$ be a small extension. If there exists a deformation $(X_A,\F_A)\to (A,Q)$ of $(X,\F)\to(\ast,Q)$, then 
\begin{enumerate}
    \item if $T_\F$ is locally free, then there is a canonical obstruction $$\eta\in \H^2\big(\underline{X},\mathcal L^\ast(\F)\big)\otimes J$$ such that $\eta=0$ is and only if there exists lifting $(X_{A'},\F_{A'})\to A'$,
    \item if an extension to $A'$ exists, then the group $$\H^1\big(\underline{X},\mathcal L^\ast(\F)\big)\otimes J$$ acts freely and transitively on the set of equivalence classes of extensions, and
    \item the automorphism group of $(X_{A'},\F_{A'})$ relative to $A$ is isomorphic to $$\H^0\big(\underline{X},\mathcal L^\ast(\F)\big)\otimes J.$$
\end{enumerate}
\end{theorem}

\begin{proof} We will only show the first claim. The second item follows directly from Proposition \ref{prop:logdef}, while the computation of the automorphism groups is a consequence of Theorem \ref{prop:logdefX} and Lemma \ref{lemma:Liederivative}.

Let $(X_A,\F_A)\to (A,Q)$ be a deformation of $(X,\F)\to (\ast,Q)$, and let $\U=\{U_i\}$ be a Stein cover of $\underline{X}$. We will start by emulating the classical construction of the obstructions to deforming a smooth variety: by the first item in Proposition \ref{prop:logdefX}, there exist local liftings $X_{A',i}\to (A',Q)$ of $X_{A'}\vert_{U_i}$ and unique isomorphisms $\varphi_{ij}:X_{A',i}\vert_{U_{ij}}\to X_{A',j}\vert_{U_{ij}}$ satisfying $\varphi_{ij}^{-1}=\varphi_{ji}$. On the triple intersections, this yields an automorphism of $X_{A',i}$ of the form 
\begin{equation}\label{cocycle1}
1+\theta_{ijk} = \varphi_{ki}\varphi_{jk}\varphi_{ij}
\end{equation}

for some $\{\theta_{ijk}\}\in C^2(\U,T_{X/(\ast,Q)}\otimes J)$.

Being $T_\F$ locally free, we can construct local extensions  $T_{\F_{A',i}}$ of $T_{\F_A}\vert_{U_i}$ and local isomorphisms  $\psi_{ij}: T_{\F_{A'},i}\rightarrow T_{\F_{A'},j}$ 
with  $\psi_{ij}^{-1} = \psi_{ji}$ and yielding an automorphism of the form 
\begin{equation}\label{eq:beta}
    1+\beta_{ijk} = \psi_{ki}\psi_{jk}\psi_{ij}
\end{equation}
for some $\beta_{ijk}\in \mathcal \sEnd(T_\F)\otimes J$, on $U_{ijk}$. We will now proceed with the elements of the proof of Proposition \ref{lema:defLie}. If we choose  liftings (using again the fact that $T_\F$ is locally free)
$f_{A',i}:T_{\F_{A'},i}\rightarrow T_{X_{A'},i}$, then comparing them on the overlaps we get a cocycle 
\begin{equation}
   \label{cocycle2} \{\overline{g}_{ij}\}=\big\{\overline{d\varphi_{ji}f_{A',j}\psi_{ij}-f_{A',i}}\big\}\in C^1\big(\U,\sHom(T_{\F},N_{\F})\big)\otimes J.
\end{equation}

If
$\{\,,\}_i: \wedge^2 T_{\F_{A',i}}\rightarrow T_{\F_{A',i}}$ are arbitrary liftings of the Lie bracket on $T_{\F_{A,i}}$ we know that the obstruction to extending $T_{\F_{A,i}}\subseteq T_{X_{A,i}/(A,Q)}$ to a subalgebra of  $T_{X_{A',i}/(A',Q)}$  corresponds to the Chevalley-Eilenberg cohomology class of 
\begin{equation}
    \label{cocycle3}
\big\{\overline{b}_i(x,y)\big\} = \big\{\overline{[f_{A',i}(x),f_{A',i}(y)]_{A'}-f_{A',i}\{x,y\}_i}\big\}\in C^0\big(\U,\sHom(\wedge^2T_\F,N_\F)\big)\otimes J,
\end{equation}
where by a slight abuse of notation we are using the same notation for the elements $x,y\in T_{X_{A,i}/(A,Q)}$ and their restriction to $X$. We claim that the cocycles defined in (\ref{cocycle1}), (\ref{cocycle2}) and (\ref{cocycle3}) define an element of $\H^2(X,\mathcal L^\bullet(\F))$, which is precisely the obstruction to globally extending to $A'$. 

Let us first see that this indeed defines a cohomology element. This is, we need to check that
\begin{enumerate}
    \item $\hat{\delta}(\{\theta_{ijk}\}])=0$, 
    \item $\hat{\delta}(\{\overline{g}_{ij}\})= -\delta(\{\theta_{ijk}\})$,
    \item $\hat{\delta}(\{\overline{b}_i\})=\delta(\{\overline{g}_{ij}\})$, and
    \item $\delta(\{\overline{b}_i\})=0$.
\end{enumerate}
In fairness, we should also check that this element in $\H^2(\underline{X},\mathcal{L}^\bullet(\F))$ does not depend on the choices of liftings made along the way, but in order to make a clearer exposition we will leave this as an exercise to the reader.
The first of these equations follows verbatim to the case of smooth schemes \cite{sernesi2007deformations}, while the last one is explicitly done in the proof of Proposition \ref{lema:defLie}. For the second identity, observe that $\hat{\delta}(\overline{g}_{ij})$ is the projection onto $N_\F$ of 
\begin{align*}
g_{ij}+d\varphi_{ji}g_{jk}\psi_{ij}-g_{ik} & = (d\varphi_{ji}f_{A',j}\psi_{ij}-f_{A',i})+d\varphi_{ji}(d\varphi_{kj}f_{A',k}\psi_{jk}-f_{A',j})\psi_{ij}-g_{ik}\\
& = d\varphi_{ji}d\varphi_{kj}f_{A',k}\psi_{jk}\psi_{ij}-(f_{A',i}+g_{ik})\\
&= (1-[\theta_{ijk},-])d\varphi_{ki}f_{A',k}\psi_{ik}(1+\beta_{ijk})-(f_{A',i}+g_{ik}),\\
 & =(1-[\theta_{ijk},-])(f_{A',i}+g_{ik})(1+\beta_{ijk})-(f_{A',i}+g_{ik})\\
& = -[\theta_{ijk},f_{A',i}(-)]+f_{A',i}\beta_{ijk},
\end{align*}
where we only used Equations (\ref{cocycle1}), (\ref{eq:beta}) and (\ref{cocycle2}) repeatedly. Hence, $\hat{\delta}(\{\overline{g}_{ij}\})=-\delta(\{\theta_{ijk}\})$ as claimed.
 Similarly, for the third equality the element $\hat{\delta}(\{\overline{b}_i\})(x,y)$ can be computed as the projection of 
{\small
\begin{align*}
d\varphi_{ji}b_j(\psi_{ij}(x),\psi_{ij}(y))&- b_i(x,y)=d\varphi_{ji}[f_{A',j}\psi_{ij}(x),f_{A',j}\psi_{ij}(y)]_{A'}- d\varphi_{ji}f_{A',j}\{\psi_{ij}(x),\psi_{ij}(y)\}_j\\
&\qquad\qquad\qquad\qquad\qquad\qquad\qquad\,\,- [f_{A',i}(x),f_{A',i}(y)]_{A'}+f_{A',i}\{x,y\}_{A',i}\\
    & = [d\varphi_{ji}f_{A',j}\psi_{ij}(x),d\varphi_{ij}f_{A',j}\psi_{ij}(y)]_{A'}- d\varphi_{ji}f_{A',j}\{\psi_{ij}(x),\psi_{ij}(y)\}_j\\
    &\qquad\qquad\qquad\qquad\qquad\qquad\qquad\;\;\; - [f_{A',i}(x),f_{A',i}(y)]_{A'}+f_{A',i}\{x,y\}_{A',i}\\
    & = [(f_{A',i}+g_{ij})(x),(f_{A',i}+g_{ij})(y)]_{A'} - (f_{A',i}+g_{ij})\psi_{ji}\{\psi_{ij}(x),\psi_{ij}(y)\}_j\\
    &\qquad\qquad\qquad\qquad\qquad\qquad\qquad\;\;- [f_{A',i}(x),f_{A',i}(y)]_{A'}+f_{A',i}\{x,y\}_{A',i}\\
    & = [f_{A',i}(x),g_{ij}(y)]_{A'}-[f_{A',i}(y),g_{ij}(x)]_{A'} \\
    &\qquad\qquad\;-(f_{A',i}+g_{ij})\big(\psi_{ji}\{\psi_{ij}(x),\psi_{ij}(y)\}_{A',j}-\{x,y\}_{A',i}\big) - g_{ij}\{x,y\}_{A',i}\\
    & = [f_{A',i}(x),g_{ij}(y)]_{A'}-[f_{A',i}(y),g_{ij}(x)]_{A'}\\
    &\qquad\qquad\qquad\qquad\quad- f_{A',i}(\psi_{ji}\{\psi_{ij}(x),\psi_{ij}(y)\}-\{x,y\}_{A',i})- g_{ij}\{x,y\}_{A',i}
\end{align*}}
\noindent where we used Equations (\ref{cocycle1}), (\ref{cocycle2}),  (\ref{cocycle3}) and that $\psi_{ji}\{\psi_{ij}(X),\psi_{ij}(Y)\}_{A',j}-\{X,Y\}_{A',i}$ belongs to $T_{\F_{A'},i}\otimes J$, hence it vanishes we evaluate it on $g_{ij}: T_{\F_{A'},i}\rightarrow T_{A',i}\otimes J$.

In summary,  we constructed an element $[\{\theta_{ijk}\},\{\overline{g}_{ij}\}, \{\overline{b}_i\}]\in \H^2(\underline{X},\mathcal{L}^\bullet(\F))$ that vanishes if the deformation of $(X_A\F_A)$ lifts to a deformation $(X_{A'},\F_{A'})$ over $A'$. Let us argue that the opposite implication is also true. Suppose that this class is zero, so we can find pair $(\{\rho_{ij}\},\{\overline{h}_i\})$ such that
\[
\{\theta_{ijk}\} = \hat\delta(\{\rho_{ij}\}),\quad \delta(\{\rho_{ij}\}) = \hat\delta(\{\overline{h}_{i}\}),\quad\text{and}\quad \delta(\{\overline{h}_i\}) = \{\overline{b}_i\}.
\]
Then, under the modifications $\varphi_{ij}\mapsto \varphi_{ij} + \rho_{ij}$ and $f_{A',i}\mapsto f_{A',i} + h_i$, it easily follows that this collection of data yields the desired lifting (see Equation (\ref{cocycle1}) and the proof of Lemma \ref{lema:defLie}).
\end{proof}

\begin{remark}
  Although the hypothesis on $T_\F$ being locally free is indeed restrictive, it is innocuous from the point of view of the main question of this paper, namely the existence of smoothings. Indeed, if a foliation $\underline{\F}$ on a normal crossings variety $\underline{X}$ admits such a smoothing, then this condition is automatically satisfied. 
\end{remark}

\begin{remark} \label{remark:LES}
    Just as in \cite[Theorem 1.6]{gomez1988transverse}, in the case of where $\dim(\F)=1$ the spectral sequence associated to the double complex of $\check{C}$-cochains of the log-leaf complex gives rise to a long exact sequence
\begin{align*}
    H^0(X,T_X)\to H^0(X,\sHom(T_\F,N_\F))\to \H^1(X,\mathcal{L}^\bullet(\F))\to H^1(X,T_X)\to \\ 
    \to H^1(X,\sHom(T_\F,N_\F))\to \H^2(X,\mathcal{L}^\bullet(\F))\to H^2(X,T_X)\to \cdots
\end{align*} 
If $H^1(X,\sHom(T_\F,N_\F))=0$ then by Theorem \ref{teo:defobst} and Proposition \ref{prop:logdefX}, we see that $Def^{ls}_{(X,F)}$ is less (hence equally) obstructed than the functor of log smooth deformations of $X\to (\ast,Q)$. In particular, if $X$
admits a formal deformation then so does $(X,\F)$.

On the other hand, it is important to point out that that hypothesis on $T_\F$ being locally free was indeed used in the construction of the obstructions. A priori, there is no guarantee that even the sheaf $T_{\F_A}$ itself extends to $X_A'$. These extra obstructions should appear at the level of $\sExt^1(T_\F,N_\F)$. The leaf complex does not seem to perceive these subtleties, and therefore we do not expect it to provide an obstruction space in full generality.
\end{remark}

We end the analysis of the log-smooth deformations of $(X,\F)\to (\ast,Q)$ by showing the existence of versal hulls in the sense of \cite{schlessinger1968functors} (i.e., a formal ring $R\in  \widehat{\mathcal A}_Q$ and a smooth morphism $h_R\to Def^{ls}_{(X,\F)}$ inducing an isomorphism at the level of tangent spaces).

\begin{theorem}
    If $\underline{X}$ is proper, then functor  $Def^{ls}_{(X,\F)}$ admits a versal hull.
\end{theorem}
\begin{proof}
    We will do this by showing that $Def^{ls}_{(X,\F)}$ satisfies conditions the hypotheses of Schlessinger's representability theorem \cite[Theorem 2.11]{schlessinger1968functors}. We will begin by noting that the groups $\H^i(X,\mathcal L^\bullet(\F))$ are finite-dimensional since they can be computed through a spectral sequence whose second page consists of the cohomology groups of the (coherent) sheaves appearing in the leaf complex.

In order to show that condition $H_1$ (resp. $H_2$) holds, we need to check that if $(X_{A'},\F_{A'})\to (A',Q)$ and $(X_{A''},\F_{A''})\to (A'',Q)$ (resp. take $A'=\C[\varepsilon]/(\varepsilon^2)$)
are smooth liftings of $(X_{A},\F_A)\to (A,Q)$, then there exists a common (resp. unique) lifting over $(B=A'\times_A A'',Q)$. Since by \cite[Theorem 8.7]{kato1996dlog_def} the corresponding non-foliated functor admits a hull, we know that these conditions hold for the log-smooth maps only.

Let $X_B\to (B,Q)$ be a (resp. the unique) lifting. We are left to analyze the existence of a lifting $\F_B$. To do this, observe that by construction the coherent sheaf $T_{X_B/(B,Q)}$ restricts to the corresponding relative tangent sheaves over $A'$ and $A''$. By \cite[Theor\'eme 1]{pourcin1969theoreme}, the functor assigning to each $f:S\to B$ the set of subsheaves of $f^*T_{X_B/(B,Q)}$ on $f^*X_B$ with flat cokernel and restricting to $T_\F$ on the central fiber is representable, and hence conditions $H_1$ and $H_2$ holds in this context, finishing the proof.
\end{proof}

\begin{corollary}\label{cor:smoothing}
Let $X$ be a proper variety with normal crossings and $(X,\F)\to (\ast,\N)$ be a log smooth morphism, with $T_{\underline{\F}}$ locally free. If there exists a formal deformation $(\widehat{X},\widehat{\F})\to (\Spec(\C[[t]]),\N)$ then $(X,\F)$  admits a smoothing. In particular, if $(\underline{X},\underline{\F})$ is $d$-semistable and $\H^2(X,\mathcal L^\bullet(\F))=0$ for the corresponding logarithmic structure, then it admits a smoothing. 
\end{corollary}
\begin{proof}
Let us denote $\X\to K$ the Kuranishi space of $X$, and $(\X,\mathcal H)\to D$ the universal subsheaf of $T_{\X/B}$ mentioned above. 
Since the formal deformation is log smooth, for every $n>0$ we have inclusions 
$$ T_{\F_n} \subseteq T_{X_n/(\Spec(\C[t]/(t^n)),\N)}\subseteq T_{X_n/\Spec(\C[t]/(t^n))}.$$
As a consequence, if we forget the logarithmic structures then the data of $\widehat{\underline{X}}$ and $\widehat{\underline{\F}}$ yields a morphism  $\Spec(\C[[t]])\to D$. By Artin's Approximation Theorem \cite[Theorem 1.2]{artin1968solutions}, there exists a convergent family $(X_\Delta,\F_\Delta)\to \Delta$ coinciding with the formal family at first order, i.e. a smoothing of $(\underline{X},\underline{\F})$. 
The second claim is straightforward from Theorem \ref{teo:defobst}.
\end{proof}

\begin{example}    
Let $\underline{X} = X_1 \sqcup_{\P^1} X_2$ be the gluing of two $\P^1$ bundles $X_i\to C_i$ along a fiber $D$, where $C$ is a smooth curve.
Let $\F_1$ and $\F_2$ the foliations induced by the rulings of $X_1$ and $X_2$ respectively. The pushout foliation $\F = \F_1\sqcup_{\P^1} \F_2$ is d-semistable thanks to Proposition \ref{prop: linear_hol}.  Let us denote $X$ the corresponding logarithmic variety, smooth of the log point. Observe that Lemma \ref{T_F_loc_free} implies that $T_\F$ coincides with the line bundle $T_{\F_1}\times_D T_{\F_2}$.  \

Gluing the curves $C_i$ along the distinguished points we get a  $\P^1$-bundle $\pi: X\rightarrow C = C_1\sqcup_p C_2$ inducing the foliation $\F$, with log-normal sheaf $\pi^\ast T_{C/(\ast,\N)}$.
Let us analyze when $H^1\big(X,\sHom(T_\F, T_{X/(\ast,\N)}/T_\F)\big) = 0$. By Remark \ref{remark:LES} this will imply that the obstruction to extending a log smooth deformation of $(X,\F)$ maps injectively to the obstruction to only deforming $X$, and hence $\F$ will deform along any formal deformation of $X$. Since $X$ admits a smoothing \cite[Chapter V]{persson1977degenerations}, then by Corollary \ref{cor:smoothing} this will guarantee the existence of a smoothing of $(\underline{X},\underline{\F})$.

Observe that by adjunction
$$
    H^1\big(X,\sHom(T_\F, T_{X/(\ast,\N)}/T_\F)\big) = H^1\big(C,\sHom(\pi_\ast T_{X/C},T_{C/(\ast,\N)})\big),
$$
where in the last equality we used that $R\pi_\ast \OO_X = \OO_C$.  Since $C$ is Cohen-Macaulay, the sheaf $\sHom(\pi_\ast T_{X/C},T_{C/(\ast,\N)})$ is locally free and $\Omega^1_{C/(\ast,\N)}$ is a dualizing sheaf for $C$ \cite[Proposition 7.70]{felten2025global}. With this in mind, Serre-Grothendieck duality implies that
\[
 H^1\big(C,\sHom(\pi_\ast T_{X/C},T_{C/(\ast,\N)})\big) = H^0\big(C,\pi_\ast T_{X/C}\otimes (\Omega^1_{C/(\ast,\N)})^{\otimes 2}\big)^\vee.
\]
A global section of this bundle corresponds to global sections of its restrictions to the components $X_1$ and $X_2$ that coincide along $D$. These new sheaves are $\pi_*T_{X_i/C_i}\otimes \Omega^1_{C_i}(\log \,p)^{\otimes 2}$, and it may be the case that they do not admit global sections.

If, for instance, the fibrations are Hirzebruch surfaces $X_i=\P(\OO_{\P^1}\oplus \OO_{\P^1}(n))\to \P^1$, from the Euler exact sequence of a projective bundle follows that 
$$\pi_\ast T_{X_i/\P^1} \simeq \sEnd(E)/\OO_{C_i}\operatorname{id}_E = \OO_{\P^1}\oplus  \OO_{\P^1}(n) \oplus \OO_{\P^1}(-n).$$ But then 
$$\pi_*T_{X_i/C_i}\otimes \Omega^1_C(\log \,p)^{\otimes 2}\simeq \OO_{\P^1}(-2) \oplus \OO_{\P^1}(n-2)\oplus \OO_{\P^1}(-n-2),$$
showing that the desired cohomology group vanishes whenever $n=0,1$, i.e., when $X_i$ is either $\P^1\times \P^1$ or $Bl_p\P^2$. 
\end{example}

Based on its non-foliated counterparts, we believe that the notions developed in this work are indeed necessary to approach existence of smoothings for normal crossings foliated pairs. However, in order to give a full answer to this issue one should restrict the analysis to a special family, for instance log Calabi-Yau foliations by means of the $T^1$-lifting technique \cite{kawamata1994logarithmic}. Hopefully, this will be the subject of future work.

    \bibliographystyle{alpha}
    \bibliography{bibliography.bib}

\end{document}

%% file: packages.tex
\usepackage[utf8]{inputenc} \usepackage[english]{babel} \usepackage[reqno]{amsmath} \usepackage{amsfonts,amssymb,amsthm} \usepackage{mathtools} \usepackage{amstext} \usepackage{amscd} 
\usepackage[T1]{fontenc} \usepackage{thmtools} \usepackage{enumitem} 
\usepackage{palatino} \usepackage{amsthm} 
\usepackage{euler} \usepackage{calrsfs} \usepackage{tikz-cd} \usepackage{graphicx} \usepackage{relsize} \usepackage{multicol} \usepackage[margin=1.6in]{geometry} \usepackage{accents} 
\usepackage{bm} \usepackage{dirtytalk} \usepackage{lipsum} \usepackage{tcolorbox} \usepackage{tikz-cd} \usetikzlibrary{patterns,decorations.pathreplacing,babel} \usetikzlibrary{shapes,intersections} 
\usepackage{tabularx} \usepackage{tikz-cd} 
\usepackage{hyperref} \hypersetup{ colorlinks=true, linkcolor=purple, filecolor=magenta, 
} 
\usepackage{xcolor} \usepackage{listings} \usepackage{inconsolata} \usepackage{float} \usepackage[labelfont=bf]{caption} 
\usepackage[square,numbers]{natbib} 
\usepackage{parskip} \usepackage[official]{eurosym} \usepackage{todonotes} \usepackage{csquotes} \usepackage{tocbasic}

%% file: commands.tex
\newcommand\C{\mathbb{C}}
\newcommand\N{\mathbb{N}}
\newcommand\Z{\mathbb{Z}}

\newcommand\Q{\mathbb{Q}}

\renewcommand\H{\mathbb{H}}

\newcommand\m{\mathfrak{m}}

\renewcommand\P{\mathbb{P}}
\newcommand{\T}{\mathbb{T}}

\newcommand{\OO}{\mathcal{O}}
\newcommand{\gp}{\mathrm{gp}}

\DeclareMathOperator{\Spec}{Spec}

\DeclareMathOperator{\Res}{Res}

\newcommand{\X}{\mathcal{X}}

\newcommand{\F}{\mathscr{F}}

\newcommand{\M}{\mathcal{M}}
\newcommand{\U}{\mathcal{U}}

\renewcommand{\H}{\mathbb{H}}
\renewcommand\v{\raise0.9ex\hbox{$\scriptscriptstyle\vee$}}

\DeclareMathOperator{\Hom}{Hom}

\DeclareMathOperator{\Ext}{Ext}
\DeclareMathOperator{\Aut}{Aut}
\DeclareMathOperator{\Pic}{Pic}

\DeclareFontFamily{OT1}{pzc}{}
\DeclareFontShape{OT1}{pzc}{m}{it}{<-> s * [1.100] pzcmi7t}{}
\DeclareMathAlphabet{\mathpzc}{OT1}{pzc}{m}{it}

\DeclareMathOperator{\sHom}{\mathpzc{Hom}}
\DeclareMathOperator{\sEnd}{\mathpzc{End}}
\DeclareMathOperator{\sExt}{\mathpzc{Ext}}

\usepackage{hyperref}

%% file: theoremstyle.tex
\newtheorem{theorem}{Theorem}[section]
\newtheorem{proposition}[theorem]{Proposition}
\newtheorem{lemma}[theorem]{Lemma}
\newtheorem{corollary}[theorem]{Corollary}

\theoremstyle{definition}

\newtheorem{definition}[theorem]{Definition}
\newtheorem{remark}[theorem]{Remark}

\newtheorem{example}[theorem]{Example}

\newtheoremstyle{named}{}{}{\itshape}{}{\bfseries}{.}{.5em}{\thmnote{#3}}

\theoremstyle{named}

%% file: bibliography.bib
@incollection{pereira2024closed,
  title={Closed meromorphic 1-forms},
  author={Pereira, Jorge Vit{\'o}rio},
  booktitle={Handbook of Geometry and Topology of Singularities V: Foliations},
  pages={447--499},
  year={2024},
  publisher={Springer}
}

@book{felten2025global,
  title={Global logarithmic deformation theory},
  author={Felten, Simon},
  year={2025},
  publisher={Springer}
}

@book{camacho1987pontos,
  title={Pontos singulares de equa{\c{c}}{\~o}es diferenciais anal{\'\i}ticas},
  author={Camacho, C{\'e}sar and Sad, Paulo},
  volume={16},
  year={1987},
  publisher={Instituto de Matem{\'a}tica Pura e Aplicada}
}

@article{deligne1970equations,
  title={Equations diff{\'e}rentielles {\`a} points singuliers r{\'e}guliers},
  author={Deligne, Pierre},
  journal={Lecture Notes in Math},
  volume={163},
  year={1970}
}

@article{schlessinger1985lie,
  title={The Lie algebra structure of tangent cohomology and deformation theory},
  author={Schlessinger, Michael and Stasheff, James},
  journal={Journal of Pure and Applied Algebra},
  volume={38},
  number={2-3},
  pages={313--322},
  year={1985},
  publisher={Elsevier}
}

@article{kawamata1994logarithmic,
  title={Logarithmic deformations of normal crossing varieties and smoothing of degenerate {C}alabi-{Y}au varieties},
  author={Kawamata, Yujiro and Namikawa, Yoshinori},
  journal={Inventiones mathematicae},
  volume={118},
  number={1},
  pages={395--409},
  year={1994},
  publisher={Springer-Verlag Berlin/Heidelberg}
}

@article{olsson2003universal,
  title={Universal log structures on semi-stable varieties},
  author={Olsson, Martin C.},
  journal={Tohoku Mathematical Journal, Second Series},
  volume={55},
  number={3},
  pages={397--438},
  year={2003},
  publisher={Mathematical Institute, Tohoku University}
}

@article{kato1988logarithmic_str,
  title={Logarithmic structures of fontaine-illusie},
  author={Kato, Kazuya},
  journal={Algebraic Analysis, Geometry and Number Theory},
  pages={191--224},
  year={1988}
}

@article{kato1996dlog_def,
  title={Log smooth deformation theory},
  author={Kato, Fumiharu},
  journal={Tohoku Mathematical Journal, Second Series},
  volume={48},
  number={3},
  pages={317--354},
  year={1996},
  publisher={Mathematical Institute, Tohoku University}
}

@book{ogus2018lectures,
  title={Lectures on logarithmic algebraic geometry},
  author={Ogus, Arthur},
  volume={178},
  year={2018},
  publisher={Cambridge University Press}
}

@article{abramovich2010logarithmic,
  title={Logarithmic geometry and moduli},
  author={Abramovich, Dan and Chen, Qile and Gillam, Danny and Huang, Yuhao and Olsson, Martin and Satriano, Matthew and Sun, Shenghao},
  journal={arXiv preprint arXiv:1006.5870},
  year={2010}
}

@article{friedman1983global,
  title={Global smoothings of varieties with normal crossings},
  author={Friedman, Robert},
  journal={Annals of Mathematics},
  volume={118},
  number={1},
  pages={75--114},
  year={1983},
  publisher={JSTOR}
}

@article{Quallbrunn_2015,
   title={Families of distributions and Pfaff systems under duality},
   journal={Journal of Singularities},
   author={Quallbrunn, Federico},
   year={2015}
}

@book{sernesi2007deformations,
  title={Deformations of algebraic schemes},
  author={Sernesi, Edoardo},
  volume={334},
  year={2007},
  publisher={Springer Science \& Business Media}
}

@book{huybrechts2010geometry,
  title={The geometry of moduli spaces of sheaves},
  author={Huybrechts, Daniel and Lehn, Manfred},
  year={2010},
  publisher={Cambridge University Press}
}

@article{gomez1988transverse,
  title={The transverse dynamics of a holomorphic flow},
  author={G{\'o}mez-Mont, Xavier},
  journal={Annals of mathematics},
  volume={127},
  number={1},
  pages={49--92},
  year={1988},
  publisher={JSTOR}
}

@book{persson1977degenerations,
  title={On degenerations of algebraic surfaces},
  author={Persson, Ulf},
  volume={189},
  year={1977},
  publisher={American Mathematical Soc.}
}

@article{artin1968solutions,
  title={On the solutions of analytic equations},
  author={Artin, Michael},
  journal={Inventiones mathematicae},
  volume={5},
  number={4},
  pages={277--291},
  year={1968},
  publisher={Springer}
}

@article{pourcin1969theoreme,
  title={Th{\'e}or{\`e}me de Douady au-dessus de $ S$},
  author={Pourcin, Genevi{\`e}ve},
  journal={Annali della Scuola Normale Superiore di Pisa-Scienze Fisiche e Matematiche},
  volume={23},
  number={3},
  pages={451--459},
  year={1969}
}

@article{schlessinger1968functors,
  title={Functors of Artin rings},
  author={Schlessinger, Michael},
  journal={Transactions of the American Mathematical Society},
  volume={130},
  number={2},
  pages={208--222},
  year={1968},
  publisher={JSTOR}
}

@article{spicer2025moduli,
  title={On moduli of foliated surfaces},
  author={Spicer, Calum and Svaldi, Roberto and Velazquez, Sebastian},
  journal={arXiv preprint arXiv:2511.13491},
  year={2025}
}

@article{calvo1994irreducible,
  title={Irreducible components of the space of holomorphic foliations},
  author={Calvo-Andrade, Omegar},
  journal={Mathematische Annalen},
  volume={299},
  number={1},
  pages={751--767},
  year={1994},
  publisher={Springer}
}

@article{ferrand2003conducteur,
  title={Conducteur, descente et pincement},
  author={Ferrand, Daniel},
  journal={Bulletin de la Société Mathématique de France},
  volume={131},
  number={4},
  pages={553--585},
  year={2003}
}
